\documentclass{amsart}
\usepackage{verbatim}
\usepackage{color}
\usepackage{amsmath}
\usepackage[colorlinks, linkcolor=blue, citecolor=blue]{hyperref}

\allowdisplaybreaks[4]
\newtheorem{theorem}{Theorem}[section]

\newtheorem{lemma}[theorem]{Lemma}
\newtheorem{Assumption}{Assumption}[section]
\theoremstyle{definition}
\newtheorem{definition}[theorem]{Definition}

\theoremstyle{remark}
\newtheorem{remark}[theorem]{Remark}

\numberwithin{equation}{section}

\newcommand{\dd}{\mathrm{d}}

\begin{document}

\title{Large deviations of reflected weakly interacting particle systems}

\author{Ping Chen}
\address{School of Mathematical Sciences, University of Science and Technology of China, Hefei, 230026, China.}
\email{chenping@mail.ustc.edu.cn}
%\thanks{This work is partly supported by the National Natural Science Foundation of China (No. 12131019, No. 11971456, No. 11671132, No. 11721101, No. 11871184) and the Fundamental Research Funds for the Central Universities (No. WK3470000024, No. WK0010000076).}

\author{Rong Wei}
\address{School of Mathematics and Statistics, Anhui Normal University, Wuhu 241002, China.}
\email{rongw@.ahnu.edu.cn}
% \thanks{Support information for the second author.}

\author{Tusheng Zhang}
\address{Department of Mathematics, University of Manchester, Oxford Road, Manchester M13 9PL, United Kingdom}
\email{tusheng.zhang@manchester.ac.uk}
% \thanks{Support information for the second author.}

%    General info

%\date{July 21, 2022}

% \dedicatory{This paper is dedicated to our advisors.}


\begin{abstract}
In this paper, we prove a large deviation principle for the empirical measures of a system of weakly interacting diffusion with reflection. We adopt the weak convergence approach.  To make this approach work, we show that the sequence of empirical measures of the controlled reflected system will converge to the weak solution of an associated reflected McKean--Vlasov equation.
\end{abstract}

\maketitle

\section{Introduction}
Let $(\Omega,\mathcal{F},\mathbb{P})$ be a probability space and $\mathcal{D}$ an open domain in $\mathbb{R}^{d}$. The purpose of this paper is  to establish a large deviation principle for the empirical measures of weakly interacting particle systems, which are given by reflected stochastic differential equations (RSDEs) of the following form: for fixed $N\in\mathbb{N}$ and finite interval $[0,T]$, $i\in \{1, \ldots, N \}$,
\begin{align}
\left\{
\begin{aligned}\label{II-eq2.1}
\mathrm{d}X^{i,N}(t)&=b(t,X^{i,N}(t),\mu^{N}(t))\mathrm{d}t+\sigma(t,X^{i,N}(t),\mu^{N}(t))\mathrm{d}W^{i}(t)\\
& \ \ \ \ -\mathrm{d}K^{i,N}(t),\\
|K^{i,N}|(t)&=\int^{t}_{0}\mathbf{1}_{\partial \mathcal{D}}(X^{i,N}(s))\mathrm{d}|K^{i,N}|(s),\\
K^{i,N}(t)&=\int^{t}_{0}\mathbf{n}(X^{i,N}(s))\mathrm{d}|K^{i,N}|(s),\\
X^{i,N}(0)&=x^{i,N},\quad X^{i,N}(t)\in\overline{\mathcal{D}},\quad t\in[0,T]
\end{aligned}
\right.
\end{align}
    where $\mathbf{n}(x)$ is the outward unit normal vector at $x\in\partial{\mathcal{D}}$, $\{W^{i},1\leq i\leq N\}$ are independent standard $d_{1}$-dimensional Wiener processes and $K^{i,N}$ is a bounded variation process with variation $|K^{i,N}|$ acting as a local time that constrains the process $X^{i, N}$ to the domain $\overline{\mathcal{D}}$. Here, $x^{i,N}\in\overline{\mathcal{D}}$ and
$$\mu^{N}(t,\omega):=\frac{1}{N}\sum_{i=1}^{N}\delta_{X^{i,N}(t,\omega)},\quad \omega\in\Omega,$$
is the empirical measure of $(X^{1,N},\ldots,X^{N,N})$ at time $t\in [0,T]$.
For the precise conditions on $\mathcal{D}$, $b$ and $\sigma$, we refer the readers to Section \ref{sec:main_result}.

Large deviations of weakly interacting diffusion can characteristic the rate of convergence of the particle system converging to a Mckean--Vlasov stochastic differential equations. The pioneer work \cite{DG} for large deviation of weakly interacting diffusion  considers a system of uniformly non-degenerate diffusion with interacting in the drift and establishes a large deviation principle for the empirical measures using discretization arguments and exponential estimates. There is now a number of  papers on the large deviations of weakly interacting diffusions for a variety of models, including  multilevel large deviations \cite{DG1,DS}, jump diffusions \cite{L1,L2}, discrete-time system \cite{DM,DG3} and singular interaction \cite{AG,HHMO}, interaction with common noise \cite{BC}.

However, there is no study so far on the large deviation of empirical measures of weakly interacting diffusion with reflection. The aim of this work is to fill  the gap.  Our work is mainly inspired by the work \cite{BudhirajaDF}. We will use the weak convergence methods, no longer requiring any time or space discretization of the system and exponential probability estimates, along the same lines of the approaches in \cite{BudhirajaDF} where the authors considered weakly interacting diffusion with no reflection. As in  \cite{BudhirajaDF}, we allow the driving noise to be  degenerate and also allow both the drift and the diffusion coefficients to depend on interactions.

The early work to study  McKean--Vlasov stochastic differential equations (SDEs) with reflecting boundaries goes back to  \cite{SznitamnA}, in which the pathwise well-posedness and propagation of chaos are proved. McKean--Vlasov SDEs are SDEs which are related to the mean-field interacting diffusions and where the coefficients depend on the law of the solution.
In a recent work \cite{CDFGGM}, the authors study a system, which is the combination of the
interaction keeping the average position prescribed, and the reflection at the boundaries. And they show pathwise well-posedness for the McKean--Vlasov SDE and the propagation of chaos. There is also some work on the small noise large deviations of reflected McKean--Vlasov SDEs, see \cite{AdamsReisRavaille,WYZ}.

With the weak convergence approach, we reduce the proof of the large deviation principle to the study of asymptotic properties of certain controlled versions of the original system with reflection. The main ingredients in our argument  are to characterize the weak limits of the controlled processes. The most closed work to ours is \cite{BudhirajaDF}, which considers large deviation property of weakly interacting systems of stochastic differential equations.
Although we both use the weak convergence approach,  there are major  differences due to the reflection on the boundary of the domain. Particularly,  the method of obtaining the Laplace upper bound in \cite{BudhirajaDF} no longer works for our model due to the reflection. Instead, we fully exploit the close relationship between stochastic differential equations with reflecting boundary and the so-called submartingale problems.
\vskip 0.3cm
The remainder of this paper is organized as follows. In Section \ref{sec:main_result}, we describe the  framework, present  the main result of the paper, and then outline how the result will be proved. In Section \ref{AC}, we consider the sub-martingale characterization of a weak solution of reflected stochastic differential equations, which will be used in the proof. Finally, we complete the proof of the main result  by showing the Laplace upper and lower bounds in Section \ref{sec:variational}.
\vskip 0.4cm
We end this section with some standard notations.

%, as well as a collection $\{W^{i},i\in\mathbb{N}\}$ of independent standard $d_1$-dimensional $(\mathcal{F}_{t})$-Winer processes.
Let $\overline{\mathcal{D}}$ be the closure of $\mathcal{D}$. Let $(\Omega,\mathcal{F}, (\mathcal{F}_{t})_{t\in [0,T]},\mathbb{P})$ be a filterted probability space where $(\mathcal{F}_{t})_{t\in[0,T]}$ satisfies the usual conditions.
Let $\mathcal{X}:=C([0,T];\overline{\mathcal{D}})$ and $\mathcal{W}:=C([0,T];\mathbb{R}^{d_1})$, equipped with maximum norm, which is denoted by $\|\cdot\|_\infty$. Let $|\cdot|$ be the Euclidean norm and $\| \cdot \|$ stand for the Hilbert--Schmidt norm, i.e.,
$\|\sigma\|^2 :=  \sum_{i=1}^{d} {\sum_{j=1}^{d_1}{\sigma_{ij}^2}}$
for any $d \times d_1$-matrix $\sigma = (\sigma_{ij}) \in \mathbb{R}^{d \times d_1} $. Given a Polish space $(E,d_{E})$, define $\mathcal{B}(E)$ as the Borel $\sigma$-field, denote by $\mathcal{P}(E)$ the space of probability measures on $\mathcal{B}(E)$, which is equipped with the topology of weak convergence. A convenient metric on this space is the bounded Lipschitz metric, which is denoted by
$$\Pi_{E}(\mu,\nu):=\sup_{f \in \mbox{Lip}_{1}(E)}\int_{E}f\mathrm{d}(\mu-\nu), \ \ \mu,\nu \in \mathcal{P}(E),$$
 where $\mbox{Lip}_{1}(E):=\{ f \in C(E): \sup_{x\in E}|f(x)|\leq 1, \sup_{x\neq y\in E}\frac{|f(x)-f(y)|}{d_{E}(x,y)}\leq 1\}$.
 Let $C_b(E)$ be the set of real-valued bounded continuous functions on $E$.
Let $C^{1,2}([0,T]\times \overline{\mathcal{D}})$ denote the set of real-valued functions, whose elements are continuously differentiable once with respect to the time variable and twice with respect to the space variable.
% which equipped with the norm
%$$\|f\|_{\infty,T}:=|f\|_{\infty}+\|f_t\|_{\infty}+ \sum_{i=1}^{d}\|f_{x_i}\|_{\infty} +\sum_{i,j=1}^{d}\| f_{x_{i}x_{j}}\|_{\infty}$$
%making it a separable metric space.
For $f\in C^{1,2}([0,T]\times\overline{\mathcal{D}})$, $\nabla_xf$ denotes the gradient of $f$ with respect to the spatial variable $x$.
Analogously, we can define $C^{1,2,2}([0,T]\times\overline{\mathcal{D}}\times \mathbb{R}^{d_1} )$, $C^{1,2,2}_{b}([0,T]\times\overline{\mathcal{D}}\times \mathbb{R}^{d_1} )$, $C^{1,2,2}_{0}([0,T]\times\overline{\mathcal{D}}\times \mathbb{R}^{d_1} )$.

\section{Main result}\label{sec:main_result}

Let $\mathcal{R}$ denote the set of all positive measures on $\mathcal{B}(\mathbb{R}^{d_{1}}\times[0,T])$, say $r$, such that $r(\mathbb{R}^{d_{1}}\times[0,t])=t$ for all $ t\in[0,T]$, which will be  the space of all deterministic relaxed controls on $\mathbb{R}^{d_{1}}\times[0,T]$. We also denote
\begin{eqnarray*}
\mathcal{R}_{1}&:=&\Big\{r\in\mathcal{R}:\int_{\mathbb{R}^{d_{1}}\times[0,T]}|y|r(\mathrm{d}y\times\mathrm{d}t)<\infty \Big\},
\end{eqnarray*}
as the space of all deterministic relaxed controls on $\mathbb{R}^{d_{1}}\times[0,T]$ with finite first moments.
We equip $\mathcal{R}$ and $\mathcal{R}_{1}$ with the topology of weak convergence of measures and weak convergence of measures plus convergence of the first moments respectively, which turns $\mathcal{R}$ and $\mathcal{R}_{1}$ into Polish spaces. If $r\in\mathcal{R}$ and $B\in\mathcal{B}(\mathbb{R}^{d_{1}})$, the mapping $[0,T]\ni t\mapsto r(B\times[0,t])$ will be absolutely continuous, hence differentiable almost everywhere. Since $\mathcal{B}(\mathbb{R}^{d_{1}})$ is countably generated, the time derivative of $r$, denoted by $r_t$, exists almost everywhere and is a measurable mapping
from $[0,T]$ to $\mathcal{P}(\mathbb{R}^{d_{1}})$, such that
$r(\mathrm{d}y\times\mathrm{d}t)=r_{t}(\mathrm{d}y)\mathrm{d}t$.

\begin{comment}
The topology of weak convergence of measures turns $\mathcal{R}$ in with to a Polish space.
Define
\begin{equation*}
\mathcal{R}_{1}:=\{r\in\mathcal{R}:\int_{\mathbb{R}^{d_{1}}\times[0,T]}|y|r(\mathrm{d}y\times\mathrm{d}t)<\infty\}.
\end{equation*}
as the  of determined relaxed controls with finite first moments equpped with the topology of weak convergence of measures plus convergence of the first momentss, which turns $\mathcal{R}_{1}$ into a Polish space, cf.\cite[Section 6.3]{Rachev}.
\end{comment}

Let $b:[0,T]\times\overline{\mathcal{D}}\times\mathcal{P}(\overline{\mathcal{D}})\to\mathbb{R}^{d}$ and $\sigma:[0,T]\times\overline{\mathcal{D}}\times\mathcal{P}(\overline{\mathcal{D}})\to\mathbb{R}^{d\times d_{1}}$ be measurable mappings.
Given a Borel measurable mapping $\nu:[0,T] \to\mathcal{P}(\overline{\mathcal{D}})$ and an adapted $\mathcal{R}_1$-valued random variable $\rho$.
We will consider the controlled RSDEs:
\begin{align}\label{II-eq2.5}
\left\{
\begin{aligned}
 \mathrm{d}\bar{X}(t)&=b(t,\bar{X}(t),\nu(t))\mathrm{d}t+\int_{\mathbb{R}^{d_{1}}}\sigma(t,\bar{X}(t),\nu(t))y\rho_{t}(\mathrm{d}y)\mathrm{d}t\\
& \ \ \ \ +\sigma(t,\bar{X}(t),\nu(t))\mathrm{d}W(t)-\mathrm{d}\bar{K}(t),\\
 |\bar{K}|(t)&=\int^{t}_{0}\mathbf{1}_{\partial \mathcal{D}}(\bar{X}(s))\mathrm{d}|\bar{K}|(s),\quad \bar{K}(t)=\int^{t}_{0}\mathbf{n}(\bar{X}(s))\mathrm{d}|\bar{K}|(s),\\
\nu(0)&=\mbox{Law}(\bar{X}(0)),\quad \bar{X}(t)\in\overline{\mathcal{D}}, \quad t \in [0,T],\\
\end{aligned}
\right.
\end{align}
where $W$ is a $d_{1}$-dimensional $\mathcal{F}_{t}$-adapted standard Wiener process. The above equation is the controlled analogue of the following reflected McKean-Vlasov SDEs:
\begin{align}\label{II-eq2.2}
\left\{
\begin{aligned}
& \mathrm{d}X(t)=b(X(t),\mbox{Law}(X(t)))\mathrm{d}t+\sigma(X(t),\mbox{Law}(X(t)))\mathrm{d}W(t)-\mathrm{d}K(t), \\
& |K(t)|=\int^{t}_{0}\mathbf{1}_{\partial \mathcal{D}}(X(s))\mathrm{d}|K|(s),\quad K(t)=\int^{t}_{0}\mathbf{n}(X(s))\mathrm{d}|K|(s),\\
& \mbox{Law}(X(0))=\nu_{0},\quad X(t)\in\overline{\mathcal{D}}, \quad t\in[0,T],
\end{aligned}
\right.
\end{align}
where $W$ is a $d_{1}$-dimensional $(\mathcal{F}_{t})$-adapted standard  Wiener process.

Recall the system given by \eqref{II-eq2.1}: for fixed $N\in\mathbb{N}$ and finite interval $[0,T]$, $i\in \{1, \ldots, N \}$,
\begin{align*}
\left\{
\begin{aligned}\label{II-eq2.2-1}
\mathrm{d}X^{i,N}(t)&=b(t,X^{i,N}(t),\mu^{N}(t))\mathrm{d}t+\sigma(t,X^{i,N}(t),\mu^{N}(t))\mathrm{d}W^{i}(t)\\
& \ \ \ \ -\mathrm{d}K^{i,N}(t),\\
|K^{i,N}|(t)&=\int^{t}_{0}\mathbf{1}_{\partial \mathcal{D}}(X^{i,N}(s))\mathrm{d}|K^{i,N}|(s),\\
K^{i,N}(t)&=\int^{t}_{0}\mathbf{n}(X^{i,N}(s))\mathrm{d}|K^{i,N}|(s),\\
X^{i,N}(0)&=x^{i,N},\quad X^{i,N}(t)\in\overline{\mathcal{D}},\quad t\in[0,T].
\end{aligned}
\right.
\end{align*}

 Define $$\mu^{N}(\cdot,\omega):=\frac{1}{N}\sum_{i=1}^{N}\delta_{X^{i,N}(\cdot,\omega)}, \quad \omega\in\Omega$$
as the empirical measure of $(X^{1,N},\ldots,X^{N,N})$ over the path space $\mathcal{X}:=C([0,T];\overline{\mathcal{D}})$.
For any $t\in [0,T]$, $\mu^{N}$ and $\mu^{N}(t)$ have relationship as
$$\mu^{N}(t)=\mu^{N}\circ\pi^{-1}_{t},$$
where $\pi_{t}:\mathcal{X} \to \overline{\mathcal{D}}$ is a projection map at time $t$. It is known that under mild conditions on the coefficients the empirical measures $\mu^{N}$ converges to the law of the solution to reflected McKean-Vlasov SDE \eqref{II-eq2.2}.

Let $\mathcal{Z}:=\mathcal{X}\times\mathcal{R}_{1}\times \mathcal{W}$. For any $z\in \mathcal{Z}$, we write it as $(\phi,r,w)$ with the understanding that $\phi\in \mathcal{X}$, $r \in \mathcal{R}_{1}$, and $w\in \mathcal{W}$.
\begin{comment}Notice that we include $\mathcal{W}$ as a component of our canonical space $\mathcal{Z}$. This will allow identification of the joint distribution of the control and driving Wiener process.
\end{comment}
If the triple $(\bar{X},\rho,W)$ defined on some filtered probability space $(\tilde{\Omega},\tilde{\mathcal{F}},(\tilde{\mathcal{F}}_{t}),\tilde{\mathbb{P}}
)$ solves equation \eqref{II-eq2.5} for
some measurable $\nu:[0,T]
\to\mathcal{P}(\overline{\mathcal{D}})$, then the distribution of $(\bar{X},\rho,W)$ under
$\tilde{\mathbb{P}}$ is an element of $\mathcal{P}(\mathcal{Z})$ and is called a weak solution of  equation \eqref{II-eq2.5}. For any $\Theta\in\mathcal{P}(\mathcal{Z})$, define  $\nu_{\Theta}:[0,T]\to\mathcal{P}(\overline{\mathcal{D}})$ as
\begin{equation}\label{II-eq2.6}
    \nu_{\Theta}(t)(B):=\Theta\big(\{(\phi,\rho,w)\in\mathcal{Z}:
    \phi(t)\in B\}\big),\quad B\in\mathcal{B}(\overline{\mathcal{D}}),\quad t\in[0,T],
\end{equation}
which is the distribution under $
\Theta$ of the first component of $Z$ at time $t$.
Note that if $\Theta$ is a weak solution of equation \eqref{II-eq2.5} with $\nu_{\Theta}(t)=\nu(t),\ \forall t \in [0,T]$, then $\Theta$ is a weak solution of the controlled RSDEs
\begin{align}\label{II-eq2.7}
\left\{
\begin{aligned}
&\mathrm{d}\bar{X}(t)=b(t,\bar{X}(t),\nu_{\Theta}(t))\mathrm{d}t+\int_{\mathbb{R}^{d_{1}}}\sigma(t,\bar{X}(t),\nu_{\Theta}(t))y\rho_{t}(\mathrm{d}y))\mathrm{d}t\\
&\quad\quad\quad\quad+\sigma(t,\bar{X}(t),\nu_{\Theta}(t)))\mathrm{d}W(t)-\mathrm{d}\bar{K}(t),\\
& |\bar{K}|(t)=\int^{t}_{0}\mathbf{1}_{\partial \mathcal{D}}(\bar{X}(s))\mathrm{d}|\bar{K}|(s),\quad \bar{K}(t)=\int^{t}_{0}\mathbf{n}(\bar{X}(s))\mathrm{d}|\bar{K}|(s),\\
&\nu_{\Theta}(t)=\mbox{Law}(\bar{X}(t)),\quad \bar{X}(t)\in\overline{\mathcal{D}},\quad t\in[0,T].\\
\end{aligned}
\right.
\end{align}

For a probability measure $\Theta \in  \mathcal{P}(\mathcal {Z})$, let $\Theta_{ \mathcal {X}}$, $\Theta_{ \mathcal {R}}$ and $\Theta_{ \mathcal {\mathcal{W}}}$ denote the first, second and the third marginal measures, respectively. Let $\mathcal {P}_{\infty}$ be the set of all probability measures $\Theta \in \mathcal {P}(\mathcal {Z}) $ such that
\begin{itemize}
\item[(i)]
\[\int_{\mathcal {R}_{1}}\int_{\mathbb{R}^{d_1}\times [0,T]}|y|^2r(\dd y \times \dd t)\Theta_{\mathcal{R}}(\dd r)<\infty.\]
\item[(ii)] $\Theta$ is a weak solution to equation \eqref{II-eq2.7}.\\
\item[(iii)] $\nu_{\Theta}(0)=\nu_{0}$, where $\nu_{0}\in \mathcal {P}(\overline{\mathcal{D}})$ is the initial distribution.
%in Assumption \ref{Ass2.1}(A2).
\end{itemize}

%%Defined canonical filtration $(\mathcal{G}_{t})$ in  $\mathcal{B}(\mathcal{Z})$ as
%%$$\mathcal{G}_{t}:=\sigma((\bar{X}(s),\rho(s),W(s)):0\leq s\leq t),\quad t\in[0,T].$$
%%The processes triple $(\bar{X},\rho,W)$  can be given explicitly as coordinate processes on $(\mathcal{Z},\mathcal{B}(\mathcal{Z}),\Theta)$ endowed with $(\mathcal{G}_{t})$ in term of
%%\begin{equation}
    %%\bar{X}(t,(\phi,r,w)):=\phi(t),\quad \rho(t,(\phi,r,w)):=r_{|\mathcal{B}(\mathbb{R}^{d_{1}}\times[0,t])},\quad W(t,(\phi,r,w)):=w(t).
%%\end{equation}

%%Throughout this paper, for any $\Theta$ we let $\Theta_{ \mathcal {X}}$, $\Theta_{ \mathcal {R}}$ denote the first and second marginal respectively.

% \begin{comment}If $(\bar{X},\rho,W)$ corresponds to a weak solution of equation \eqref{II-eq2.7}) on $(\tilde{\Omega},\tilde{\mathcal{F}},\tilde{\mathbb{P}},
% \tilde{\mathcal{F}}_{t})$, then $\tilde{P}\circ(\tilde{X},\rho,W)^{-1}\in\mathcal{P}(\mathcal{Z})$. Therefore the property of weak uniqueness can be formulated in terms of probability on $\mathcal{B}(\mathcal{Z})$.
%\end{comment}

Define the map $\bar{v}: \mathcal{Z}
\to \mathcal{Z}^{0}:=\overline{\mathcal{D}} \times \mathcal{R}_{1} \times \mathcal{W}$ as $\bar{v}(\phi, r, w)=(\phi(0), r, w)$.
\begin{definition}[Weak Uniqueness]
\label{def2.1}
We say that the solution of equation \eqref{II-eq2.7} is weakly unique if for any $\Theta$ and $\tilde{\Theta} \in \mathcal{P}_\infty$ such that $\Theta \circ \bar{v}^{-1}= \tilde{\Theta}\circ \bar{v}^{-1}$,
%$\tilde{\Theta} \in \mathcal{P}_\infty$ and $\Theta|_{\mathcal{B}(\mathcal{R}_{1}\times \mathcal{W})}=\tilde{\Theta}|_{\mathcal{B}(\mathcal{R}_{1}\times \mathcal{W})}$,
we have $\Theta = \tilde{\Theta}$.
\end{definition}

Now we are going to state the precise assumption and the main result.

\begin{definition}[Rate Function \cite{DZ,DupuisEllis}]
A function $I: \mathcal{P}(\mathcal{X}) \to [0,\infty]$ is called a rate function on $\mathcal{P}(\mathcal{X})$, if for any $C<\infty$, the level set $\{\theta\in \mathcal{P}(\mathcal{X}): I(\theta)\leq C\}$ is compact.
\end{definition}

\begin{definition}[Large Deviation Principle \cite{DZ,DupuisEllis}]
Let $I$ be a rate function on $\mathcal{P}(\mathcal{X})$. The sequence $\{\mu^{N},N\in\mathbb{N}\}$ is said to satisfy large deviation principle on $\mathcal{P}(\mathcal{X})$ with the rate function $I$ if, for all Borel subset $\Gamma$ of $\mathcal{P}(\mathcal{X})$,
\begin{equation*}
-\inf_{\theta\in \Gamma^{o}}I(\theta)\leq\liminf_{N\to\infty}\frac{1}{N}\log \mathbb{P}\{\mu^{N}\in\Gamma\}\leq\limsup_{N\to\infty}\frac{1}{N}\log \mathbb{P}\{\mu^{N}\in\Gamma\}\leq -\inf_{\theta\in\overline{\Gamma}}I(\theta),
\end{equation*}
where $\Gamma^{o}$ is the interior of $\Gamma$.
\end{definition}

\begin{definition}[Laplace Principle \cite{DupuisEllis}]
Let $I$ be a rate function on $\mathcal{P}(\mathcal{X})$. The sequence $\{\mu^{N},N\in\mathbb{N}\}$ is said to satisfy the Laplace principle on $\mathcal{P}(\mathcal{X})$ with rate function $I$ if for all bounded continuous function $F:\mathcal{P}(\mathcal{X})\to\mathbb{R}$,
\begin{equation*}\label{II-eq2.9}
\lim_{N\to\infty}-\frac{1}{N}\log\mathbb{E}\big\{\exp[-N\cdot  F(\mu^{N})]\big\}=\inf_{\theta\in\mathcal{P}(\mathcal{X})}\{F(\theta)+I(\theta)\}.
\end{equation*}
\end{definition}
It is known that the Laplace principle holds in the above setting if and only if $\{\mu^{N},N\in\mathbb{N}\}$ satisfies a large deviation principle with rate function $I$ (see \cite[Section 1.2]{DupuisEllis}).

Let us make the following assumptions with respect to the domain, coefficients $b$, $\sigma$, and the family $\{x^{i,N}\}\subset\overline{\mathcal{D}}$ of initial conditions:
\begin{Assumption}\label{Ass2.1}
\begin{itemize}
\item [(A1)]\label{assum:1} $\mathcal{D} $ is a bounded, convex, smooth domain in $\mathbb{R}^{d}$.
\item [(A2)]There exists $\nu_{0}\in\mathcal{P}(\overline{\mathcal{D}})$ such that for all $\nu_{0}$-integrable $f:\overline{\mathcal{D}} \rightarrow \mathbb{R}$,
\[ \lim_{N \rightarrow \infty}\frac{1}{N}\sum_{i=1}^{N}f(x^{i,N}) = \int_{\overline{\mathcal{D}}}f(x)\mathrm{d}\nu_{0}(x). \]
%For some $\nu_{0}\in\mathcal{P}(\overline{\mathcal{D}})$, $\frac{1}{N}\sum^{N}_{i=1}\delta_{x^{i,N}}\to \nu_{0}$ as $N$ tends to infinity.
\item [(A3)] Let $b:[0,T]\times\overline{\mathcal{D}}\times\mathcal{P}(\overline{\mathcal{D}})\to\mathbb{R}^{d}$ and $\sigma:[0,T]\times\overline{\mathcal{D}}\times\mathcal{P}(\overline{\mathcal{D}})\to\mathbb{R}^{d\times d_{1}}$ be measurable and there exist constants $L$ and $K \in (0,\infty)$ such that for each $t\in [0,T], x, y \in \overline{\mathcal{D}} $ and $\mu, \nu \in \mathcal{P}(\overline{\mathcal{D}})$,
  \[ |b(t,x, \mu)| +\|\sigma(t,x, \mu) \| \leq L\]
  and
  \[|b(t,x,\mu)-b(t,y,\nu)|+\|\sigma(t,x,\mu)-\sigma(t,y,\nu)\|\leq K(|x-y|+\Pi_{\overline{\mathcal{D}}}(\mu,\nu)).\]
%\item [(A4)]Weak uniqueness of solution holds for equation \eqref{II-eq2.7}.
\item [(A4)]Weak uniqueness of solution holds for Eq.\eqref{II-eq2.7}.
\end{itemize}
\end{Assumption}

\begin{remark}
\begin{itemize}
  \item [(R1)]
  %The Assumption (A3) implies that  the coefficients $b$, $\sigma$ are uniformly continuous and uniformly bounded on $[0,T]\times\overline{\mathcal{D}}\times  \mathcal{P}(\overline{\mathcal{D}})$.
   Under Assumption~\ref{Ass2.1}(A3), it follows by standard arguments that for each $N$, equation \eqref{II-eq2.1} has a unique solution. %Similarly, it can also be shown that \eqref{II-eq2.3} has a unique solution.
   Moreover, according to Theorem 3.2 in \cite{AdamsReisRavaille}, equation \eqref{II-eq2.2} admits a unique solution.
  \item [(R2)] The assumption (A4) is satisfied if the diffusion coefficients depended only on state variables like $\sigma(t,x,\mu)=\sigma(t,x)$, or only on distributions such as $\sigma(t,x,\mu)=\sigma(t,\mu)$ (cf. \cite[Lemma 3.1] {BC}).
      \item [(R3)]\label{R1}
  %For any $x\in\partial \mathcal{D}$, we define
%\begin{eqnarray*}
%\mathcal{N}_{x}&:=\bigcup_{r>0}\{\mathbf{n}(x)\in\mathbb{R}^{d} : |\mathbf{n}(x)|=1,B(x+r\mathbf{n}(x),r)\cap\mathcal{D}=\emptyset\},
%\end{eqnarray*}
%where $B(x,r):=\{y\in\mathbb{R}^{d}:|x-y|<r\}$. Then, $\mathcal{N}_{x}$ is the collection of the outward unit normal vectors at $x$ and $\mathcal{N}_{x}$ is a singleton set.
 Since the domain $\mathcal{D}$ is convex, we have for any $y\in\mathcal{D}$ and $x\in\partial\mathcal{D}$,
\begin{equation*}
  \langle\mathbf{n}(x),y-x\rangle\leq 0.
\end{equation*}
\end{itemize}
\end{remark}

The main result is stated as follows.
\begin{theorem}\label{Thm2.5}
Let Assumption \ref{Ass2.1} hold. Then the family of empirical measures $\{\mu^{N},N\in\mathbb{N}\}$ of the solutions of the interacting reflected system \eqref{II-eq2.1} satisfies the Laplace Principle with rate function
$$I(\theta)=\inf_{\Theta\in\mathcal{P}_{\infty}:\Theta_{\mathcal{X}}=\theta}\frac{1}{2}\int_{\mathcal{R}_1}\int_{\mathbb{R}^{d_{1}}\times[0,T]}|y|^{2}r(\mathrm{d}y\times\mathrm{d}t)\Theta_{\mathcal{R}}(\mathrm{d}r)$$
where $\theta\in\mathcal{P}(\mathcal{X})$ and  $\inf\emptyset:=\infty$ by convention.
\end{theorem}

To explain the steps of the proof, we need the controlled system of RSDEs \eqref{II-eq2.1}.
For $N\in\mathbb{N}$, let
$\mathcal{H}_{N}$ be the space of all $(\mathcal{F}_{t})$-progressively measurable functions $h:[0,T]\times\Omega \to \mathbb{R}^{N\times d_{1}}$ such that
$$\mathbb{E}\left[\int^{T}_{0}|h(t)|^{2}\mathrm{d}t\right]<\infty,$$
where $h$ is written as $h=(h_{1},\ldots,h_{N})$, and $h_{i}$ is its $i$-th entry, which is $d_1$-dimensional.

Given $h\in\mathcal{H}_{N}$, we consider the following controlled system of RSDEs
\begin{align}
\left\{
\begin{aligned}\label{II-eq2.3}
& \mathrm{d}\bar{X}^{i,N}(t)=b(t,\bar{X}^{i,N}(t),\bar{\mu}^{N}(t))\mathrm{d}t+\sigma(t,\bar{X}^{i,N}(t),\bar{\mu}^{N}(t))h_{i}(t)\mathrm{d}t\\
&\qquad\quad\quad \quad+\sigma(t,\bar{X}^{i,N}(t),\bar{\mu}^{N}(t))\mathrm{d}W^{i}(t)-\mathrm{d}\bar{K}^{i,N}(t),\\
& |\bar{K}^{i,N}|(t)=\int^{t}_{0}\mathbf{1}_{\partial \mathcal{D}}(\bar{X}^{i,N}(s))\mathrm{d}|\bar{K}^{i,N}|(s),\\
& \bar{K}^{i,N}(t)=\int^{t}_{0}\mathbf{n}(\bar{X}^{i,N}(s))\mathrm{d}|\bar{K}^{i,N}|(s),\\
& \bar{X}^{i,N}(0)=x^{i,N},\quad \bar{X}^{i,N}(t)\in\overline{\mathcal{D}},\quad t\in [0,T],
\end{aligned}
\right.
\end{align}
where $\bar{\mu}^{N}(t)$ is empirical measure of $(\bar{X}^{1,N}(t),\ldots,\bar{X}^{N,N}(t))$ defined by
$$\bar{\mu}^{N}(t,\omega):=\frac{1}{N}\sum_{i=1}^{N}\delta_{\bar{X}^{i,N}(t,\omega)}, \quad \omega\in\Omega.$$
We also use $\bar{\mu}^{N}$ to denote the process version of $\bar{\mu}^{N}(t)$.

\vskip 0.4cm

The proof of Theorem \ref{Thm2.5} is based on following variational representation given by Theorem 3.6 in \cite{BudhirajaDupuis}, i.e. for any $F\in C_{b}(\mathcal{P}(\mathcal{X}))$
\begin{equation*}\label{II-eq2.10}
    -\frac{1}{N}\log\mathbb{E}\left\{\exp[-N\cdot  F(\mu^{N})]\right\}=\inf_{h^{N}\in\mathcal{H}_{N}}\Big\{\frac{1}{2}\mathbb{E}\big[\frac{1}{N}\sum^{N}_{i=1}\int^{T}_{0}|h^{N}_{i}|^{2}\mathrm{d}t\big]+\mathbb{E}\left[F(\bar{\mu}^{N})\right]\Big\},
\end{equation*}
where $\bar{\mu}^{N}$ is the empirical measure of the solution to equation \eqref{II-eq2.3} for $h^{N}\in \mathcal{H}_{N}$.
\vskip 0.3cm
According to the arguments in \cite[section 3]{BudhirajaDF}
,the Laplace principle (hence the large deviation principle) can be established through the following two steps.
\begin{itemize}
\item[{\bf Step 1:}]
We establish the Laplace principle upper bound by showing that for any sequence $\{h^{N},N\in\mathbb{N}\}$ with $h^{N}\in\mathcal{H}_{N}$,
\begin{eqnarray}\label{lowerbound}
\begin{split}
&   \liminf_{N\mapsto\infty}\Big\{\frac{1}{2}\mathbb{E}\Big[\frac{1}{N}\sum^{N}_{i=1}\int^{T}_{0}|h^{N}_{i}    |^{2}\mathrm{d}t\Big]+\mathbb{E}\big[F(\bar{\mu}^{N})\big]\Big\}\\
&   \geq\inf_{\Theta\in\mathcal{P}_{\infty}}\Big\{
    \frac{1}{2}\int_{\mathcal{R}_{1}}\int_{\mathbb{R}^{d_{1}}\times[0,T]}|y|^{2}r(\mathrm{d}y\times\mathrm{d}t)    \Theta_{\mathcal{R}_{1}}(\mathrm{d}r)+F(
    \Theta_{\mathcal{X}})\Big\}.
\end{split}
\end{eqnarray}

This will be done in Section~\ref{sec:lower_bound}.

\item[{\bf Step 2:}]We verify the Laplace principle lower bound in Section~\ref{sec:upper_bound}, by showing that for any $\Theta \in \mathcal{P}_{\infty}$ there exists a sequence $\{h^{N},N\in\mathbb{N}\}$ with $h^{N}\in\mathcal{H}_{N}$ such that
\begin{eqnarray}\label{superbound}
\begin{split}
&  \limsup_{N\to\infty}\Big\{\frac{1}{2}\mathbb{E}\big[\frac{1}{N}\sum^{N}_{i=1}\int^{T}_{0}
    |h^{N}_{i}|^{2}\mathrm{d}t\big]+\mathbb{E}[F(\bar{\mu}^{N})]\Big\}\\
& \leq
    \frac{1}{2}\int_{\mathcal{R}_{1}}\int_{\mathbb{R}^{d_{1}}\times[0,T]}|y|^{2}r(\mathrm{d}y\times\mathrm{d}t)    \Theta_{\mathcal{R}_{1}}(\mathrm{d}r)+F(
    \Theta_{\mathcal{X}}).
    \end{split}
\end{eqnarray}
\end{itemize}

\begin{comment}
\begin{remark}
According to \cite{BudhirajaDupuis}, the representation in \eqref{II-eq2.10} should apply to an infinite-dimensional Brownian motion and the infimum would be over a collection of controls indexed by $i\in\mathbb{N}$. However, since those controls with $i>N$ have no effect on $\bar{\mu}^{N}$, we can and will assume they are zero.
\end{remark}
\end{comment}

\section{Sub-martingale Problem }\label{AC}
As a preparation for the proof of the main result, in this section, we will introduce some sub-martingale problems and provide the relations to the weak solution of RSDEs \eqref{II-eq2.7}.

To begin with, we introduce the definition of the sub-martingale problem described in \cite{SV}.
\begin{comment}
Let $G$ be a non-empty open set in $\mathbb{R}^{d}$ generated by a function $\varphi$ satisfying the following assumption.
\begin{Assumption}\label{assum 3.1}
\begin{itemize}
\item [(A1)] $\varphi\in C^{2}_{b}(\mathbb{R}^{d})$,
\item [(A2)] $G=\{x\in\mathbb{R}^{d}:\varphi(x)>0\}$ and     $\partial G=\{x\in\mathbb{R}^{d}:\varphi(x)=0\}$,
\item [(A3)] $|\nabla\varphi(x)|\geq 1$ for any $x\in\partial G$.
\end{itemize}
\end{Assumption}
\end{comment}
Let $\bar{b}:[0,T] \times \overline{\mathcal{D}}\to \mathbb{R}^{d}$ be a bounded measurable function and $\bar{\sigma}:[0,T] \times \overline{\mathcal{D}}\to \mathbb{R}^{d\times d_{1}}$ a bounded continuous function.

Let $\tilde{X}$ be the canonical process on $\mathcal{X}$. Set $\mathcal{F}^{0}_t=\sigma(\tilde{X}(s),\ s \leq t )$ and define operator
$$\mathcal{L}_t:=\frac{1}{2}\sum^{d}_{i,j=1}(\bar{\sigma}\bar{\sigma}^{T})_{ij}(t,x)\frac{\partial^{2}}{\partial x_{i}\partial x_{j}}+\sum^{d}_{i=1}\bar{b}_{i}(t,x)\frac{\partial}{\partial x_{i}}.$$

\begin{definition}[Sub-martingale Problem]
We say that a probability measure $\hat{\mathbb{P}}$ on $(\mathcal{X},\mathcal{B}(\mathcal{X}))$ solves a sub-martingale problem for coefficients $\bar{b}$, $\bar{\sigma}$ and $\mathbf{n}(x)$ if
$$f(t,\tilde{X}(t))-f(0,\tilde{X}(0))-\int^{t}_{0}(f_{s}+\mathcal{L}_sf)(s,\tilde{X}(s))\mathrm{d}s$$
is a $\hat{\mathbb{P}}$-sub-martingale with respect to the canonical filtration $(\mathcal{F}^{0}_t)$ for any $f\in C^{1,2}_{0}([0,T]\times\overline{\mathcal{D}})$ satisfying
\begin{equation*}
\langle \nabla_{x}f(t,x),\mathbf{n}(x)\rangle \leq 0 \quad\mbox{on}\quad [0,T]\times\partial \mathcal{D}.
\end{equation*}
\end{definition}

Consider the following RSDEs
\begin{align}\label{III-eq3.3}
\left\{
\begin{aligned}
&   \mathrm{d}Y(t)=\bar{b}(t,Y(t))\mathrm{d}t+\bar{\sigma}(t,Y(t))
    \mathrm{d}W(t)-\mathrm{d}L(t),\\
&   |L(t)|=\int^{t}_{0}\mathbf{1}_{\partial         \mathcal{D}}
    (Y(s))\mathrm{d}|L|(s),\quad L(t)=\int^{t}_{0}\mathbf{n}(Y(s))\mathrm{d}|L|(s),\\
&   Y(t)\in\overline{\mathcal{D}}, \quad t\in[0,T],
\end{aligned}
\right.
\end{align}
where $W$ is a $d_{1}$-dimensional $\mathcal{F}_t$-adapted Wiener process.

For any $f\in C^{1,2}([0,T]\times\overline{\mathcal{D}})$, we define a real-valued process $(M_{f}(t))_{t\in[0,T]}$ on probability space $(\mathcal{X},\mathcal{B}(\mathcal{X}),\hat{\mathbb{P}})$ by:
\begin{eqnarray*}
\begin{split}
 M_{f}(t,\tilde{X}) &:= f(t,\tilde{X}(t))-f(0,\tilde{X}(0))-\int^{t}_{0}(f_{s}+\mathcal{L}_sf)(s,\tilde{X}(s))\mathrm{d}s.\\
\end{split}
\end{eqnarray*}
The next result provides the relationship between a sub-martingale problem and the weak solutions of RSDEs \eqref{III-eq3.3}.
\begin{lemma}\label{Lemma 3.3}
Assume that the measure $\hat{\mathbb{P}}\in\mathcal{P}(\mathcal{X})$, then $\hat{\mathbb{P}}$ is a weak solution of equation \eqref{III-eq3.3} if and only if $\hat{\mathbb{P}}$
solves a sub-martingale problem for coefficients $\bar{b}$, $\bar{\sigma}$ and $\mathbf{n}$, or equivalently, $M_{f}$ is a $\hat{\mathbb{P}}$-sub-martingale  with respect to the canonical filtration $(\mathcal{F}^{0}_t)$ for any $f\in C^{1,2}([0,T]\times {\overline{\mathcal{D}}})$ satisfying
$\langle \nabla_{x}f(t,x),\mathbf{n}(x)\rangle\leq 0$ on $[0,T]\times\partial{\mathcal{D}}$.
\end{lemma}

\begin{proof}
The ``$\Rightarrow$'' part is obvious by It\^o's formula.

Now we show the ``$\Leftarrow$'' part. Assume $\hat{\mathbb{P}}$ is a solution to the sub-martingale problem. By Theorem 2.4 in \cite{SV}, we know that there exists a unique, continuous, non-decreasing, adapted function $L:[0,T]\times \mathcal{X}\mapsto[0,\infty)$ such that $L(0)=0$,
$L(t)=\int^{t}_{0}\mathbf{1}_{\partial\mathcal{D}}(\tilde{X}(s))\mathrm{d}L(s)$ and for any $f\in C^{1,2}([0,T]\times{\overline{\mathcal{D}}})$
\begin{eqnarray*}\label{III-eq3.5}
\tilde{M}_{f}(t,\tilde{X}):=M_{f}(t,\tilde{X})+\int^{t}_{0}\langle\nabla_{x}f(s,\tilde{X}(s)),\mathbf{n}(\tilde{X}(s))\rangle\mathrm{d}L(s)
\end{eqnarray*}
is  a $\hat{\mathbb{P}}$-martingale.

In particular, choose $f^i(t,x) = x_{i}$ to obtain that
\begin{eqnarray*}
\begin{split}
\tilde{M}_{f^i}(t,\tilde{X})&:=\tilde{X}_{i}(t)-\tilde{X}_{i}(0)-\int^{t} _0\bar{b}_{i}(s,\tilde{X}(s))\mathrm{d}s\\
& \ \ \ \ +\int^{t}_{0}n_i(\tilde{X}(s))\mathrm{d}L(s)\\
\end{split}
\end{eqnarray*}
is a $\hat{\mathbb{P}}$-martingale. \\
Similarly, for each $i,j \in \{1,2,\ldots,d\}$, letting $f^{i,j}(t,x)=x_{i}x_{j}$, we see that
\begin{eqnarray}\label{III-eq3.9}
\begin{split}
 \tilde{M}_{f^{i,j}}(t,\tilde{X})&:=\tilde{X}_{i}(t)\tilde{X}_{j}(t)-\tilde{X}_{i}(0)\tilde{X}_{j}(0)-\int^{t}_{0}\bar{b}_{i}(s,\tilde{X}(s))\tilde{X}_{j}(s)\mathrm{d}s\\
 & \ \ \ \ -\int^{t}_{0}\bar{b}_{j}(s,\tilde{X}(s))\tilde{X}_{i}(s)\mathrm{d}s+\int^{t}_{0}\tilde{X}_{i}(s)n_{j}(\tilde{X}(s))\mathrm{d}L(s)\\
 & \ \ \ \ +\int^{t}_{0}\tilde{X}_{j}(s)n_{i}(\tilde{X}(s))\mathrm{d}L(s)-\int^{t}_{0}(\bar{\sigma}\bar{\sigma}^{T})_{ij}(s,\tilde{X}(s))\mathrm{d}s
\end{split}
\end{eqnarray}
is also a $\hat{\mathbb{P}}$-martingale. Applying It\^o's formula to $\tilde{X}_{i}(t)\tilde{X}_{j}(t)$ and comparing with \eqref{III-eq3.9}, we deduce that
\begin{equation*}
    \langle \tilde{M}_{f^i},\tilde{M}_{f^j}\rangle(t)=\int^{t}_{0}\sum^{d_{1}}_{k=1}\bar{\sigma}_{ik}\bar{\sigma}_{kj}(s,\tilde{X}(s))\mathrm{d}s.
\end{equation*}
Then, according to Theorem II.7.1' in \cite{IkedaWa}, there exists a $d_{1}$-dimensional $(\tilde{\mathcal{F}}_t)$-Wiener process $\tilde{W}=(\tilde{W}(t))$ on an extension of $(\tilde{\Omega},\tilde{\mathcal{F}},(\tilde{\mathcal{F}}_t),\tilde{\mathbb{P}})$ such that
\begin{equation*}
    \tilde{M}_{f^i}(t)=\sum^{d_{1}}_{i=1}\int^{t}_{0}\bar{\sigma}_{ik}(s,\tilde{X}(s))\mathrm{d}\tilde{W}^{k}(s),\quad i=1,2,\ldots, d.
\end{equation*}
Therefore, $(\tilde{X},\tilde{W},L)$ is a weak solution to the reflected stochastic differential equation \eqref{III-eq3.3}.
\end{proof}

In the remaining part of this section, we will show that the measure $\Theta\in\mathcal{P}(\mathcal{Z})$ is a weak solution of equation \eqref{II-eq2.7} if and only if it solves a sub-martingale problem.

%Firstly, as the same arguments in \cite[section 4]{BudhirajaDF}, we know that for any $\Theta\in\mathcal{P}(\mathcal{Z})$ the mapping $\nu_{\Theta}:[0,T]\mapsto\mathcal{P}(\mathcal{X})$ is continuous which implies that the set $\{\nu_{\Theta}:t\in[0,T]\}$ is a compact in $\mathcal{P}(\overline{\mathcal{D}})$.
\vskip 0.4cm

Recall that $\mathcal{Z}=\mathcal{X}\times\mathcal{R}_{1}\times \mathcal{W}$ and $\nu_{\Theta} $ is given by \eqref{II-eq2.6}. Let $(\bar{X},\rho,W)$ be the canonical process on $\mathcal{Z}$, namely
\begin{equation*}\label{X}
  \bar{X}(t,(\phi,r,w)):=\phi(t),\ \rho(t,(\phi,r,w)):=r_{|\mathcal{B}(\mathbb{R}^{d_{1}}\times[0,t])},\ W(t,(\phi,r,w)):=w(t),
\end{equation*}
and $(\mathcal{G}_{t})$ the canonical filtration in  $\mathcal{B}(\mathcal{Z})$ defined as
\[\mathcal{G}_{t}:=\sigma((\bar{X}(s),\rho(s),W(s)):0\leq s\leq t),\quad t\in[0,T].\]
Let $\Theta\in\mathcal{P}(\mathcal{Z})$. Given $f\in C^{1,2,2}_{0}([0,T]\times{\overline{\mathcal{D}}}\times\mathbb{R}^{d_{1}})$,
we define a real-valued process $(M^{\Theta}_{f}(t))_{t\in[0,T]}$ on probability space $(\mathcal{Z},\mathcal{B}(\mathcal{Z}),\Theta)$ as
\begin{eqnarray}
M^{\Theta}_{f}(t,(\phi,r,w))
&:=&f(t,\phi(t),w(t))-f(0,\phi(0),w(0))-\int^{t}_0\frac{\partial f}{\partial s}(s,\phi(s),w(s))\dd s\nonumber\\
& &-\int^{t}_{0}\int_{\mathbb{R}^{d_{1}}}\mathcal{A}^{\Theta}_{s}(f)(s,\phi(s),y,w(s))r_{s}(\mathrm{d}y) \mathrm{d}s, \label{III-eq3.1}
\end{eqnarray}
where for any $s\in[0,T]$, $x\in {\overline{\mathcal{D}}}$, $y, z\in\mathbb{R}^{d_{1}}$, $\mathcal{A}^{\Theta}_{s}(f)$ is defined as
\begin{eqnarray}\label{III-eq3.2}
\begin{split}
    \mathcal{A}^{\Theta}_{s}(f)(s,x,y,z)
&:=  \langle b(s,x,\nu_{\Theta}(s))+\sigma(s,x,\nu_{\Theta}(s))y,
    \nabla_{x}f(s,x,z)\rangle\\
&\quad +  \frac{1}{2}\sum^{d}_{i,j=1}(\sigma\sigma^{T})_{ij}(s,x,\nu_{\Theta}(s))\frac{\partial^{2}f}{\partial {x_{i}}\partial {x_{j}}}(s,x,z)\\
&\quad +  \sum^{d}_{i=1}\sum^{d_1}_{j=1}\sigma_{ij}(s,x,\nu_{\Theta}(s))\frac{\partial^{2}f}{\partial {x_{i}}\partial {z_{j}}}(s,x,z)\\
&\quad +\frac{1}{2}\sum^{d_1}_{i=1}\frac{\partial^{2}f}{\partial {z_{i}}\partial {z_{i}}}(s,x,z).
\end{split}
\end{eqnarray}

The last result of this section is stated as follows.
\begin{theorem}\label{Thm 3.1}
Let measure $\Theta\in\mathcal{P}(\mathcal{Z})$
satisfy $\Theta(\{(\phi,r,w)\in\mathcal{Z}:w(0)=0\})=1$.
Then $\Theta$ is a weak solution of equation \eqref{II-eq2.7} if and only if %$\Theta$
%solves a sub-martingale problem for coefficients $b$, $\sigma$ and $\mathbf{n}$, or equivalently
$M^{\Theta}_{f}$ is a sub-martingale under $\Theta$ with respect to the canonical filtration $(\mathcal{G}_{t})$ for all $f\in C^{1,2,2}_{0}([0,T]\times{\overline{\mathcal{D}}}\times \mathbb{R}^{d_{1}})$ with $\langle \nabla_{x}f(t,x,z),\mathbf{n}(x)\rangle\leq 0$ on $[0,T]\times\partial{\mathcal{D}}\times\mathbb{R}^{d_{1}}$.
\end{theorem}
Before proving Theorem \ref{Thm 3.1}, we state the following Lemmas whose proofs are similar to those of Lemma 2.2 and Theorem 2.4 in \cite{SV}.
\begin{lemma}\label{submartingale1}
Suppose that $M^{\Theta}_{f}$ is a $\Theta$-sub-martingale for any $f\in C^{1,2,2}_{0}([0,T]\times {\overline{\mathcal{D}}} \times \mathbb{R}^{d_{1}})$ satisfying
$\langle \nabla_{x}f(t,x,z),\mathbf{n}(x) \rangle \leq 0 $ on $[0,T] \times \partial{\mathcal{D}} \times \mathbb{R}^{d_{1}} $. Then, for each $f\in C^{1,2,2}_{b}([0,T]\times{\overline{\mathcal{D}}} \times \mathbb{R}^{d_{1}})$ with $\langle \nabla_{x}f(t,x,z),\mathbf{n}(x) \rangle \leq 0$ on $[0,T] \times \partial{\mathcal{D}} \times \mathbb{R}^{d_{1}}$, $M^{\Theta}_{f}(t)$ is a $\Theta$-local-sub-martingale.
\end{lemma}
\begin{proof}
Assume that $f\in C^{1,2,2}_{b}([0,T] \times {\overline{\mathcal{D}}} \times \mathbb{R}^{d_{1}})$ with
$\langle \nabla_{x}f(t,x,z), \mathbf{n}(x) \rangle \leq 0$ on $[0,T] \times \partial{\mathcal{D}} \times \mathbb{R}^{d_{1}}$. For each $n \geq 1$, choose $\eta_{n} \in C_{0}^{\infty}(\mathbb{R}^{d_1})$ such that
$0 \leq \eta_{n} \leq 1$, $\eta_{n}=1$ on $\{z \in \mathbb{R}^{d_1}: |z| \leq n\}$ and all derivatives of $\eta_{n}$ up to the second order are uniformly bounded on $n$. Let
\[f_n = \eta_n \cdot f. \]
Then $f_n \in C^{1,2,2}_{0}([0,T] \times {\overline{\mathcal{D}}} \times \mathbb{R}^{d_{1}}) $ and $\langle \nabla_{x}f_n(t,x,z),\mathbf{n}(x) \rangle \leq 0$ on $[0,T] \times \partial{\mathcal{D}} \times \mathbb{R}^{d_{1}}$. So $M^{\Theta}_{f_n}(t)$ is a $\Theta$-sub-martingale.
For each $M \in \mathbb{N}$, define a stopping time
\[\tau_{M}((\phi, r, w)):=\inf\{t \in [0,T]: \int_{\mathbb{R}^{d_1}\times [0,t]}|y|r(\mathrm{d}y \times \mathrm{d}s)\geq M \}. \]
Then $M^{\Theta}_{f_n}(t \wedge \tau_{M} )$ is a $\Theta$-sub-martingale.
Obviously, $M^{\Theta}_{f_n}(t \wedge \tau_{M} ) \rightarrow M^{\Theta}_{f}(t \wedge \tau_{M} )$ boundedly and so $M^{\Theta}_{f}(t \wedge \tau_{M} )$ is a $\Theta$-sub-martingale. It follows that $M^{\Theta}_{f}(t)$ is a $\Theta$-local-sub-martingale.
\end{proof}

\begin{lemma}\label{submartingale2}
There exists a continuous, non-decreasing and adapted function $\xi: [0,T] \times \mathcal{Z} \mapsto [0,\infty)$ such that $\xi(0)=0$,
\[\xi(t)= \int_{0}^t\mathbf{1}_{\partial {D}}(\phi(s)) \mathrm{d}\xi(s),\]
and for all
$f \in  C^{1,2,2}_{b}([0,T] \times \overline{\mathcal{D}} \times \mathbb{R}^{d_{1}})$,
\[M^{\Theta}_{f} + \int_0^t \langle\nabla_{x}f(s, \phi(s), w(s)),\mathbf{n}(\phi(s))\rangle \mathrm{d}\xi(s)\]
is a $\Theta$-local-martingale.
\end{lemma}
The proof of this lemma is a minor modification of the proof of Theorem 2.4 in \cite{SV}. We omit the details.
\vskip 0.4cm
Now we come back to the proof of Theorem \ref{Thm 3.1}
\begin{proof}
The ``$\Rightarrow$'' part is obvious by It\^o's formula.

Now we show the ``$\Leftarrow$'' part. From Lemma \ref{submartingale2}, we know that there exists a continuous, non-decreasing and adapted function
$\xi:[0,T] \times \mathcal{Z} \mapsto [0,\infty)$ such that $\xi(0)=0$,
$\xi(t)= \int_0^t\mathbf{1}_{\overline{\mathcal{D}}}(\phi(s)) \mathrm{d}\xi(s)$,
and for all
$f \in  C^{1,2,2}_{b}([0,T] \times \overline{\mathcal{D}} \times \mathbb{R}^{d_{1}})$,
\begin{eqnarray*}\label{submart3}
\tilde{M}^{\Theta}_{f}(t,(\phi, r, w)):= M^{\Theta}_{f}(t,(\phi, r, w))
                                         + \int^{t}_{0}\langle\nabla_{x}f(s,\phi(s), w(s)),\mathbf{n}(\phi(s))\rangle\mathrm{d}\xi(s)
\end{eqnarray*}
is a $\Theta$-local-martingale.

In particular, for each $i \in \{1,2,\ldots,d\}$, choose $f^{i}(t,x,z)=x_{i}$ to obtain that
\begin{eqnarray*}\label{submart4}
\begin{split}
    \tilde{M}^{\Theta}_{f^{i}}(t,(\phi, r, w))
&:=  \phi_{i}(t)-\phi_{i}(0)-\int_0^t b_{i}(s,\phi(s),\nu_{\Theta}(s))\mathrm{d}s\\
&\quad -  \int_0^t \int_{\mathbb{R}^{d_1}}\sum^{d_1}_{k=1}\sigma_{ik}(s, \phi(s),\nu_{\Theta}(s) )y_{k}r_{s}(\mathrm{d}y)\mathrm{d}s\\
&\quad +  \int_0^t n_i(\phi(s))\mathrm{d}\xi(s)\\
\end{split}
\end{eqnarray*}
is a $\Theta$-local-martingale.

For each $i,j \in \{1,2,\ldots,d\}$, choosing $f^{i,j}(t,x,z)=x_{i}x_{j}$, we see that
\begin{eqnarray}\label{submart5}
\begin{split}
    \tilde{M}^{\Theta}_{f^{i,j}}(t,(\phi, r, w))
&:=  \phi_{i}(t)\phi_{j}(t)-\phi_{i}(0)\phi_{j}(0)\\
&\quad -\int_0^t b_{i}(s,\phi(s),\nu_{\Theta}(s))\phi_{j}(s)\mathrm{d}s -\int_0^t b_{j}(s,\phi(s),\nu_{\Theta}(s))\phi_{i}(s)\mathrm{d}s\\
&\quad -\int_0^t \int_{\mathbb{R}^{d_1}}\sum^{d_1}_{k=1}\sigma_{ik}(s, \phi(s),\nu_{\Theta}(s) )y_{k}\phi_{j}(s)r_{s}(\mathrm{d}y)\mathrm{d}s\\
&\quad -\int_0^t \int_{\mathbb{R}^{d_1}}\sum^{d_1}_{k=1}\sigma_{jk}(s, \phi(s),\nu_{\Theta}(s) )y_{k}\phi_{i}(s)r_{s}(\mathrm{d}y)\mathrm{d}s\\
&\quad -\int_0^t \sum^{d_{1}}_{k=1}\sigma_{ik}\sigma_{kj}(s,\phi(s),\nu_{\Theta}(s))\mathrm{d}s + \int_0^t \phi_{i}(s)n_j(\phi(s))\mathrm{d}\xi(s) \\
&\quad + \int_0^t \phi_{j}(s)n_i(\phi(s))\mathrm{d}\xi(s)\\
\end{split}
\end{eqnarray}
is also a $\Theta$-local-martingale.  Applying It\^o's formula to $\bar{X}_i(t)\bar{X}_j(t)$ and comparing with \eqref{submart5}, we deduce that
\begin{equation}\label{submart6}
\langle\tilde{M}^{\Theta}_{f^i},\tilde{M}^{\Theta}_{f^j}\rangle(t)=\int^{t}_{0}\sum^{d_{1}}_{k=1}\sigma_{ik}\sigma_{kj}(s,\phi(s),\nu_{\Theta}(s))\mathrm{d}s.
\end{equation}

For each $i \in \{1, 2, \ldots, d_{1}\}$, choose $g^{i} \in C_{0}^{1,2,2}([0,T] \times \overline{\mathcal{D}} \times \mathbb{R}^{d_{1}})$ such that
$g^{i}(t,x,z)=z_{i}$ on $[0,T] \times \overline{\mathcal{D}} \times \{z\in \mathbb{R}^{d_1}: |z| \leq n\}$ and for each $n \in \mathbb{N}$, define the stopping time $\tau_{n}= \inf\{ t \in [0,T]: |w(t)|> n\}$ to obtain that
\begin{equation*}\label{submart7}
\tilde{M}^{\Theta}_{g^{i}}(t \wedge  \tau_{n},(\phi, r, w)):= w_{i}(t\wedge \tau_{n} ) - w_{i}(0)
\end{equation*}
is a $\Theta$-martingale. It follows that $\{\tilde{M}^{\Theta}_{g^{i}}(t),t\geq 0\}$ is a $\Theta$-local-martingale.

Similarly, for each $i, j \in \{1, 2, \ldots, d_{1}\}$, choosing $g^{i,j}(t,x,z)\in C_{0}^{1,2,2}([0,T]\times\overline{\mathcal{D}}\times \mathbb{R}^{d_{1}})$
such that
$g^{i,j}(t,x,z)=z_{i}z_{j}$ on $[0,T] \times \overline{\mathcal{D}} \times \{z\in \mathbb{R}^{d_1}: |z| \leq n\}$, we have
\begin{equation}\label{submart8}
\tilde{M}^{\Theta}_{g^{i,j}}(t,(\phi, r, w)):=w_{i}(t)w_{j}(t)-w_{i}(0)w_{j}(0)-\frac{1}{2}\delta_{ij}t
\end{equation}
is a $\Theta$-local-martingale. Applying It\^o's formula to $W_{i}(t)W_{j}(t)$ and comparing with \eqref{submart8}, we deduce that
\begin{equation}\label{submart9}
\langle\tilde{M}^{\Theta}_{g^i},\tilde{M}^{\Theta}_{g^j}\rangle(t)= \langle W_{i},W_{j}\rangle(t)=\delta_{ij}t.
\end{equation}
Therefore, $W(t)$ is a $(\mathcal{G}_{t})$-Wiener process on $(\mathcal{Z},\mathcal{B}(\mathcal{Z}),(\mathcal{G}_{t}),\Theta)$.

For each $i \in \{1, 2, \ldots, d\}$, $j \in \{1, 2, \ldots, d_{1}\}$, by choosing
$h^{i,j}(t,x,z) \in C_{0}^{1,2,2}([0,T] \times\overline{\mathcal{D}} \times \mathbb{R}^{d_{1}})$
such that
$h^{i,j}(t,x,z)=x_{i}z_{j}$ on $[0,T] \times \overline{\mathcal{D}} \times \{z\in \mathbb{R}^{d_1}: |z| \leq n\}$, using a similar argument as \eqref{submart9}, we see that
\begin{equation}\label{submart10}
 \langle\tilde{M}^{\Theta}_{f^i},\tilde{M}^{\Theta}_{g^j}\rangle(t)= \int_0^t \sigma_{ij}(s, \phi(s),\nu_{\Theta}(s))\mathrm{d}s.
\end{equation}
Then, by \eqref{submart6}, \eqref{submart9} and \eqref{submart10} and using the  arguments in the proof of Theorem II.7.1' in \cite{IkedaWa}, we obtain that
\[ M^{\Theta}_{f^i}(t,(\phi, r, w))=\sum_{k=1}^{d_{1}}\int_0^t \sigma_{ik}(s,\phi(s),\nu_{\Theta}(s))\mathrm{d}W_{k}(s), i=1, 2, \ldots, d. \]
Therefore, the canonical processes $(\bar{X}, \rho, W)$ is a solution of equation \eqref{II-eq2.7}.

\end{proof}

\section{Laplace Principle}\label{sec:variational}
In this section, we will prove the upper and the lower bound of the Laplace principle.
We start by presenting an auxiliary lemma, which will be used in the proofs of the Laplace principle. Recall the controlled RSDEs \eqref{II-eq2.3},
 the empirical measure $\bar{\mu}^{N}$ and the path space  $\mathcal{X}:=C([0,T];\overline{\mathcal{D}})$.
\begin{lemma}\label{tightness}
Suppose Assumption \ref{Ass2.1} holds. Let $\{h^{N}, N \in \mathbb{N}\}$ be a sequence of elements in $\mathcal{H}_{N}$ that satisfies
\begin{equation}\label{IV-Ieq4.1}
    \sup_{N \in \mathbb{N}}\mathbb{E}\left[\frac{1}{N}\sum^{N}_{i=1}\int^{T}_{0}|h^{N}_{i}(t)|^{2}\mathrm{d}t\right]<\infty.
\end{equation}
Then the family of the laws of the $\mathcal{P}(\mathcal{X})$-valued random variables $\bar{\mu}^N$ is tight.
\end{lemma}
\begin{proof}
For positive constants $\alpha$ and $M$, we define a mapping $G_\alpha$ by
$$\mathcal{X} \ni f \mapsto G_\alpha(f) := \sup_{0\leq s < t \leq T} \frac{|f(t)-f(s)|}{|t-s|^\alpha}$$
and set
$$H_M := \Big\{ \mu \in \mathcal{P}\big(\mathcal{X}\big): \mu(G_\alpha(f)) \leq M \Big\}.$$

We claim that $H_M$ is tight (relatively compact) in $\mathcal{P}\big(\mathcal{X}\big)$. Indeed, for each $\tilde{M} \in (0, +\infty)$, since the domain $\mathcal{D}$ is bounded, we see that the set $B_{\tilde{M}} := \big\{f \in\mathcal{X}, G_\alpha(f) \leq \tilde{M}\big\}$ is relatively compact in
$\mathcal{X}$ according to the Arz\'{e}la-Ascoli theorem. On the other hand, by Chebyshev's inequality, we have
\begin{equation*}\label{T1}
\begin{aligned}
\sup_{\mu \in H_M}\mu(B^c_{\tilde{M}})& = \sup_{\mu \in H_M}\mu \big(\{f \in\mathcal{X}, G_\alpha(f) > \tilde{M}\} \big) \\
                                      & \leq \sup_{\mu \in H_M}\frac{\mu (G_\alpha(f))}{\tilde{M}}\\
                                      & \leq \frac{M}{\tilde{M}}.
\end{aligned}
\end{equation*}
Then for any $\varepsilon >0$, there exists a constant $\tilde{M}$ depending on $\varepsilon$ such that
$$\sup_{\mu \in H_M} \mu(B^c_{\tilde{M}}) \leq \varepsilon.$$
This shows that $H_M$ is relatively compact in $\mathcal{P}\big(\mathcal{X}\big)$. On the other hand, we have
\begin{equation}\label{T3}
\begin{aligned}
\sup_{N \in \mathbb{N}}\mathbb{P}(\bar \mu^N \in H^c_M) &= \sup_{N \in \mathbb{N}}\mathbb{P}(\bar \mu^N(G_\alpha(f))>M )\leq\frac{\sup_{N \in \mathbb{N}}\mathbb{E}[\bar \mu^N(G_\alpha(f))]}{M}.
\end{aligned}
\end{equation}

Next, we show that $\mathbb{E}[\bar \mu^N(G_\alpha(f))]$ is uniformly bounded with respect to $N$. From the controlled RSDE \eqref{II-eq2.3}, for any $0 \leq s \leq t \leq T$, we have
\begin{eqnarray}\label{T4}
\bar X^{i, N}(t) - \bar X^{i, N}(s)
&=& \int^t_s b(r, \bar X^{i,N}(r), \bar \mu^N(r)) \mathrm{d}r\nonumber\\
& &+ \int^t_s \sigma(r,\bar X^{i, N}(r), \bar\mu^N(r))h^N_i(r)\mathrm{d}r\nonumber\\
& &+ \int^t_s \sigma(r,\bar X^{i, N}(r), \bar \mu^N(r)) \dd W^i(r)\\
& &- \int_s^t \mathbf{n}(\bar X^{i, N}(r))\dd |\bar K^{i,N}|(r).\nonumber
\end{eqnarray}
Applying the It\^{o}'s formula to \eqref{T4}, we have
\begin{eqnarray*}\label{T5}
& &|\bar X^{i, N}(t) - \bar X^{i, N}(s)|^2 \nonumber\\
&=& 2 \int_s^t \langle \bar X^{i, N}(r) - \bar X^{i, N}(s) , b(r,\bar X^{i, N}(r), \bar\mu^N(r)) \rangle \mathrm{d}r\nonumber\\
& & + 2\int_s^t \langle \bar X^{i, N}(r) - \bar X^{i, N}(s) , \sigma(r,\bar X^{i, N}(r), \bar\mu^N(r))h^N_i(r) \rangle \mathrm{d}r\nonumber\\
& & + 2\int_s^t \langle \bar X^{i, N}(r) - \bar X^{i, N}(s) , \sigma(r,\bar X^{i, N}(r), \bar\mu^N(r))
\mathrm{d}W^i(r) \rangle\\
& & +  \int_s^t \| \sigma(r,\bar X^{i, N}(r), \bar\mu^N(r)) \|^2\mathrm{d}r\nonumber\\
& & + 2\int_s^t\langle \bar X^{i, N}(s)-\bar X^{i, N}(r) ,
\mathbf{n}(\bar X^{i, N}(r))\rangle \mathrm{d}|\bar K^{i, N}|(r) \nonumber\\
&= : & I^{s,t}_1 + I^{s,t}_2 + I^{s,t}_3 + I^{s,t}_4 + I^{s,t}_5.\nonumber
\end{eqnarray*}
Now, it follows from the definition of $G_\alpha$ that
\begin{equation*}%\label{T6}
\begin{aligned}
\mathbb{E}[G^{2}_\alpha(\bar X^{i,N}(\cdot))] &= \mathbb{E}\Big[\sup_{0 \leq s < t \leq T} \frac{|\bar X^{i, N}(t) - \bar X^{i,N}(s)|^2}{|t-s|^{2\alpha}}\Big]\\
                                  &\leq \sum_{i=1}^5\mathbb{E}\Big[\sup_{0 \leq s < t \leq T}\frac{I^{s,t}_i}{|t-s|^{2\alpha}}\Big].
\end{aligned}
\end{equation*}

Note that there exists a positive constant $\tilde{L}$ such that $\mathcal{D} \subset B(0,\tilde{L})$.

Therefore, combining Assumption \ref{Ass2.1}(A3), we obtain that
\begin{equation}\label{T7}
\mathbb{E}\Big[\sup_{0 \leq s < t \leq T}\frac{I^{s,t}_1}{|t-s|^{2\alpha}}\Big] \leq 4L\tilde{L} \mathbb{E}[\sup_{0 \leq s < t \leq T} |t-s|^{1-2\alpha}],
\end{equation}
and
\begin{equation}\label{T8}
\mathbb{E}\Big[\sup_{0 \leq s < t \leq T}\frac{I^{s,t}_4}{|t-s|^{2\alpha}}\Big] \leq L^2 \mathbb{E}[\sup_{0 \leq s < t \leq T} |t-s|^{1-2\alpha}].
\end{equation}
By H\"{o}lder inequality, we have
\begin{eqnarray}\label{T9}
& &     \mathbb{E}\Big[\sup_{0 \leq s < t \leq                      T}\frac{I^{s,t}_2}{|t-s|^{2\alpha}}\Big]\\
&\leq&  2\mathbb{E}\Big[\sup_{0 \leq s < t \leq                     T}\frac{\int_s^t |\bar X^{i, N}(r) -\bar X^{i,              N}(s)|\cdot
        |\sigma(r,\bar X^{i, N}(r), \bar\mu^N(r))h^N_i(r)|\mathrm{d}r}{|t-s|^{2\alpha}}\Big]\nonumber\\
&\leq&  4\tilde{L}\mathbb{E}\Big[\sup_{0 \leq s < t \leq T}                 \frac{(\int_s^t\| \sigma(r,\bar X^{i, N}(r),                \bar\mu^N(r))\|^2\mathrm{d}r)^{\frac{1}{2}}                 \cdot(\int_s^t|h^N_i(r)|^2\mathrm{d}r)^{\frac{1}{2}}         }{|t-s|^{2\alpha}}\Big]\nonumber\\
&\leq&  4L\tilde{L}\mathbb{E}\Big[\sup_{0 \nonumber \leq s < t \leq T}         |t-s|^{\frac{1}{2}-2\alpha} \cdot                           (\int_s^t|h^N_i(r)|^2\mathrm{d}r)^{\frac{1}{2}}\Big].\nonumber
\end{eqnarray}
Since the domain $\mathcal{D}$ is convex, we have $I^{s,t}_5
\leq 0$ and hence
\begin{equation}\label{T10}
\mathbb{E}\Big[\sup_{0 \leq s < t \leq T}\frac{I^{s,t}_5}{|t-s|^{2\alpha}}\Big] \leq 0.
\end{equation}

Set
\[Z^i(t):= \int_0^t \langle \bar X^{i, N}(r), \sigma(r,\bar X^{i, N}(r), \bar\mu^N(r))
\mathrm{d}W^i(r) \rangle, \]
and
\[Y^i(t):= \int_0^t \sigma(r,\bar X^{i, N}(r), \bar\mu^N(r))
\mathrm{d}W^i(r) \]
Then, we have
\begin{equation}\label{T10-1}
I^{s,t}_3=2(Z^i(t)-Z^i(s))-2\langle \bar X^{i, N}(s), Y^i(t)-Y^i(s)\rangle.
\end{equation}
We now proceed to give an estimate for $|Z^i(t)-Z^i(s)|$.
By the Burkholder--Davis--Gundy inequality, for each $p > 2$ and $0 \leq s \leq t \leq T$,
\begin{eqnarray*}%\label{T11}
& &\mathbb{E}[|Z^i(t) - Z^i(s) |^{2p}]\\
&=&\mathbb{E}\Big[\Big|\int_s^t \langle \bar X^{i, N}(r), \sigma(r,\bar X^{i, N}(r), \bar\mu^N(r))
\mathrm{d}W^i(r) \rangle\Big|^{2p}\Big]\nonumber\\
&\leq& C\tilde{L}^{2p}L^{2p}|t-s|^p,\nonumber
\end{eqnarray*}
where $C$ is a positive constant depending on $d_1$ and $p$. Applying the Garsia-Rodemich-Rumsey lemma (see Corollary 1.2 in \cite{WalshJohn}), there exists a random variable $A_i$ such that with probability one, for all $0 \leq s \leq t \leq T$,
\[  |Z^i(t,\omega) - Z^i(s,\omega)| \leq A_{i}(\omega)|t-s|^{\frac{p-2}{2p}}, \]
where $\mathbb{E}[A_{i}^{2p}] \leq C\tilde{L}^{2p}L^{2p}.$
By the similar arguments,
there exists also a random variable $B_i$ such that with probability one, for all $0 \leq s \leq t \leq T$,
\[  |Y^i(t,\omega) - Y^i(s,\omega)| \leq B_{i}(\omega)|t-s|^{\frac{p-2}{2p}}, \]
where $\mathbb{E}[B_{i}^{2p}] \leq CL^{2p}.$
It follows from (\ref{T10-1}) that
\begin{equation}\label{T12}
\begin{aligned}
\mathbb{E}\Big[\sup_{0 \leq s < t \leq T}\frac{I^{s,t}_3}{|t-s|^{2\alpha}}\Big] &\leq
                                                                 C\mathbb{E}\Big[\sup_{0 \leq s < t \leq T}\frac{|Z^i(t)-Z^i(s)|+|Y^i(t)-Y^i(s)|}{|t-s|^{2\alpha}}\Big]\\
                                                                         &\leq C\mathbb{E}[(A_i+B_i)\sup_{0 \leq s < t \leq T}|t-s|^{\frac{p-2}{2p}-2\alpha}].
\end{aligned}
\end{equation}
Let $p=4$ and $\alpha = \frac{1}{8}$. It follows from \eqref{T7}--\eqref{T12} that
\begin{equation*}\label{T13}
\begin{aligned}
\mathbb{E}[G^{2}_{\frac{1}{8}}(\bar X^{i,N}(\cdot))] &\leq
                                             \Big\{(4L\tilde{L} + L^2)T^{\frac{3}{4}} + C\mathbb{E}[A_{i}+B_i]+ 4L\tilde{L}T^{\frac{1}{4}}\mathbb{E}[(\int_0^T|h^N_i(r)|^2\mathrm{d}r)^{\frac{1}{2}}] \Big\}\\
                                         &\leq C\big(1+\mathbb{E}[\int_0^T|h^N_i(r)|^2\mathrm{d}r]\big),
\end{aligned}
\end{equation*}
for some positive constant $C$ independent of $N$.

Therefore, using the condition \eqref{IV-Ieq4.1}, for each $N \in \mathbb{N}$,
\begin{equation*}\label{T14}
\begin{aligned}
\mathbb{E}[\bar\mu^N(G_{\frac{1}{8}}(\cdot))] &= \frac{1}{N} \sum_{i=1}^N \mathbb{E}[G_{\frac{1}{8}}(\bar X^{i, N}(\cdot))]\\
&\leq \frac{1}{N} \sum_{i=1}^N (1+ \mathbb{E}[G^2_{\frac{1}{8}}(\bar X^{i.N}(\cdot))])\\
&\leq \tilde{C},
\end{aligned}
\end{equation*}
where $\tilde{C}$ is independent on $N$.
In combination with \eqref{T3}, it follows that for any $\varepsilon >0$, there exists a positive constant $M$ depending on $\varepsilon$ such that
\[ \sup_{N \in \mathbb{N}}\mathbb{P}(\bar \mu^N \in H^c_{M}) \leq \varepsilon .\]
This implies that the family of the laws of $\{\bar\mu^N, N \in \mathbb{N}\}$ is tight in $\mathcal{P}\big(\mathcal{P}\big(\mathcal{X}\big)\big)$.
\end{proof}

%Throughout the remainder of this work, the value of positive constant $C$ may be different from line to line.

\subsection{Laplace Upper Bound}\label{sec:lower_bound}
In this subsection, we establish the Laplace upper bound \eqref{lowerbound} in {\bf Step 1}, namely, we will show that for any sequence $\{h^{N},N\in\mathbb{N}\}$ with $h^{N}\in\mathcal{H}_{N}$,
\begin{eqnarray}\label{lowerbound-1}
\begin{split}
&   \liminf_{N\mapsto\infty}\Big\{\frac{1}{2}\mathbb{E}\Big[\frac{1}{N}\sum^{N}_{i=1}\int^{T}_{0}|h^{N}_{i}    |^{2}\mathrm{d}t\Big]+\mathbb{E}\big[F(\bar{\mu}^{N})\big]\Big\}\\
&   \geq\inf_{\Theta\in\mathcal{P}_{\infty}}\Big\{
    \frac{1}{2}\int_{\mathcal{R}_1}\int_{\mathbb{R}^{d_{1}}\times[0,T]}|y|^{2}r(\mathrm{d}y\times\mathrm{d}t)    \Theta_{\mathcal{R}}(\mathrm{d}r)+F(
    \Theta_{\mathcal{X}})\Big\}.
\end{split}
\end{eqnarray}
 For a sequence $\{h^{N} \in \mathcal{H}_{N}, N\in\mathbb{N}\}$, we may assume that
\begin{equation*}\label{IV-ieq4.1}
    \mathbb{E}\left[\frac{1}{N}\sum^{N}_{i=1}\int^{T}_{0}|h^{N}_{i}(t)|^{2}\mathrm{d}t\right]\leq 4\|F\|_{\infty},
\end{equation*}
since otherwise the inequality \eqref{lowerbound-1} is automatic.
%And further assume that  for all $\omega \in \Omega$, $N \in \mathbb{N}$, $i \in \{1,\ldots,N\}$, $h^N_i(\cdot, \omega)$ has finite first moments.

For each $N \in \mathbb{N}$, define a $\mathcal{P}(\mathcal{Z})$-valued random variable by
\begin{equation}\label{V2}
Q^N_\omega (A\times R \times B) := \frac{1}{N} \sum_{i=1}^N \delta_{\bar X^{i,N}(\cdot,\omega)}(A) \cdot \delta_{\rho^{i, N}_\omega}(R)
\cdot \delta_{ W^{i}(\cdot,\omega)}(B)
\end{equation}
for any $A\times R \times B \in \mathcal{B}(\mathcal{Z})$ and $\omega \in \Omega$, where $\bar X^{i,N}$ is the solution of equation \eqref{II-eq2.3} with $h^{N}=(h^{N}_1, ... ,h^{N}_N)$ and
\begin{equation*}\label{II-eq2.4}
\rho^{i,N}_{\omega}(B\times I):=\int_{I}\delta_{h^{N}_i(t,\omega)}(B)\mathrm{d}t,\quad B\in\mathcal{B}(\mathbb{R}^{d_{1}}), \quad I\in\mathcal{B}([0,T]),\quad \omega\in\Omega.
\end{equation*}

\begin{lemma}\label{lem1}
The family $\{Q^N,N \in \mathbb{N}\}$ of $\mathcal{P}(\mathcal{Z})$-valued random variables is tight.
\end{lemma}
\begin{proof}
By Lemma \ref{tightness}, we know that the family of the first marginals of $\{Q^N,N \in \mathbb{N}\}$  is tight. The rest of the proof is the same as that of Lemma 5.1 in \cite{BudhirajaDF}, we omit the details.
\end{proof}

\vskip0.3cm
The following result is crucial for the proof the Laplace upper bound.
\begin{theorem}\label{lem2}
Let $\{Q^{N_j},j \in \mathbb{N}\}$ be a weakly convergent subsequence of $\{Q^N,N \in \mathbb{N}\}$, with $Q$ being a  $\mathcal{P}(\mathcal{Z})$-valued random variable defined on some probability space $(\tilde{\Omega},\tilde{\mathcal {F}},\tilde{\mathbb{P}})$ such that $Q^{N_j}\stackrel{j\rightarrow\infty}{\longrightarrow} Q$ in distribution. Then, $Q_\omega \in \mathcal{P}_{\infty}$ for $\tilde{\mathbb{P}}$-almost all $\omega \in \tilde{\Omega}$, where $\mathcal{P}_{\infty}$ is the family (defined in Section 2) of weak solutions of the controlled reflected McKean-Vlasov equations:
\begin{align}\label{II-eq2.7-1}
\left\{
\begin{aligned}
&\mathrm{d}\bar{X}(t)=b(t,\bar{X}(t),\nu_{\Theta}(t))\mathrm{d}t+\int_{\mathbb{R}^{d_{1}}}\sigma(t,\bar{X}(t),\nu_{\Theta}(t))y\rho_{t}(\mathrm{d}y))\mathrm{d}t\\
&\quad\quad\quad\quad+\sigma(t,\bar{X}(t),\nu_{\Theta}(t)))\mathrm{d}W(t)-\mathrm{d}\bar{K}(t),\\
& |\bar{K}|(t)=\int^{t}_{0}\mathbf{1}_{\partial \mathcal{D}}(\bar{X}(s))\mathrm{d}|\bar{K}|(s),\quad \bar{K}(t)=\int^{t}_{0}\mathbf{n}(\bar{X}(s))\mathrm{d}|\bar{K}|(s),\\
&\nu_{\Theta}(t)=\mbox{Law}(\bar{X}(t)),\quad \bar{X}(t)\in\overline{\mathcal{D}},\quad t\in[0,T].\\
\end{aligned}
\right.
\end{align}

% \begin{comment}corresponds to a weak solution of equation \eqref{II-eq2.7} for $\tilde{\mathbb{P}}$-almost all $\omega \in \tilde{\Omega}$.\end{comment}

\end{theorem}

\begin{proof}
Set $I:= \{ N_j, j \in \mathbb{N}  \}$ and write $(Q^n)_{n \in I}$ for $(Q^{N_j})_{j \in \mathbb{N}}$, then $Q^n \rightarrow Q$ in distribution. Note that by Fatou's Lemma (see Theorem A.3.12 in \cite{DupuisEllis}) and Assumption \ref{Ass2.1}(A2), it is easy to see that for $\tilde{\mathbb{P}}$-a.s. $\omega \in \tilde{\Omega}$, $Q_\omega$ satisfies conditions (i) and (iii) in the definition of $\mathcal{P}_{\infty}$. To complete the proof, we need to prove that for $\tilde{\mathbb{P}}$-a.s. $\omega \in \tilde{\Omega}$, $Q_\omega$ is a weak solution to equation \eqref{II-eq2.7-1}

According to  Theorem \ref{Thm 3.1}, a probability measure $\Theta \in \mathcal{P}(\mathcal{Z})$ with $\Theta(\{(\phi,r,w)\in\mathcal{Z}:w(0)=0\})=1$ is a weak solution of equation \eqref{II-eq2.7-1} if for all $f\in
C^{1,2,2}_{0}([0,T]\times \overline{\mathcal{D}}\times \mathbb{R}^{d_{1}})$ with $\langle \nabla_{x}f(t,x,z),\mathbf{n}(x)\rangle\leq 0$ on $[0,T]\times\partial{\mathcal{D}}\times\mathbb{R}^{d_{1}}$, $M^{\Theta}_{f}$ (defined in \eqref{III-eq3.1} and \eqref{III-eq3.2}) is a sub-martingale under $\Theta$ with respect to the canonical filtration $(\mathcal{G}_{t})$, i.e., the following holds
\begin{equation}\label{V4''}
\mathbb{E}_{\Theta}[\Psi \cdot (M^\Theta_f(t_1 )-  M^\Theta_f(t_0 ))] \geq 0
\end{equation}
for all $t_0, t_1 \in [0,T]$ with $t_0 \leq t_1$, and $\mathcal{G}_{t_0}$-measurable $\Psi \in C_b^{+}(\mathcal{Z})$. Here $C_{b}^{+}(\mathcal{Z})$ denotes the space of nonnegative functions in $C_{b}(\mathcal{Z})$.

Note that it is enough to show that \eqref{V4''} holds for any countable collection of times $t_0$ and $t_1$ that is dense in $[0,T]$, any countable collection of $\Psi \in C_b^{+}(\mathcal{Z})$ that generates the $\sigma$-algebras $\mathcal{G}_{t_0}$, and the countable collection of test function $f$ that is dense in $C^{1,2,2}_0([0,T]\times \overline{\mathcal{D}} \times \mathbb{R}^{d_1})$ with $\langle \nabla_x f(t,x,z), n(x) \rangle \leq 0$ on $[0,T]\times\partial \mathcal{D} \times \mathbb{R}^{d_1}$. Thus, there is a countable collection $\mathcal{I} \subset [0,T]^2 \times C_b^{+}(\mathcal{Z}) \times
\{f \in C_0^{1,2,2}([0,T]\times \overline{\mathcal{D}} \times \mathbb{R}^{d_1}) : \langle \nabla_x f(t,x,z),\mathbf{n}(x) \rangle \leq 0 \ \mbox{on} \ [0,T] \times \partial\mathcal{D} \times \mathbb{R}^{d_1}\}$ of test parameters such that if \eqref{V4''} holds for all $(t_0, t_1, \Psi, f) \in \mathcal{I}$, then $\Theta$ is a weak solution of equation \eqref{II-eq2.7-1}.

To verify the sub-martingale property of $M^\Theta_f$ with $\Theta=Q_{\omega}$, $\omega \in \tilde{\Omega}$, we introduce a continuous function with compact support. For each $B\in(0,\infty)$, let $g_{B}:\mathbb{R}^{d_1}\rightarrow \mathbb{R}^{d_1}$ be a continuous function with compact support satisfying $g_{B}(y)=y$ for $|y|\leq B$ and $|g_B(y)| \leq |y|+1$ for every $y \in \mathbb{R}^{d_1}$. Also, define the operator $\mathcal{A}^{\Theta,B}_{s}$ by replacing $y$ on the right side of \eqref{III-eq3.2} with $g_B(y)$ and define $M^{\Theta,B}_f$ by replacing $\mathcal{A}^{\Theta}_s$ in \eqref{III-eq3.1} with $\mathcal{A}^{\Theta,B}_{s}$.

For each $(t_0, t_1, \Psi, f)\in \mathcal{I}$, define $\Phi = \Phi_{(t_0, t_1, \Psi, f)}$ by
\begin{equation}\label{fubu}
\mathcal{P}(\mathcal{Z}) \ni \Theta \mapsto \Phi(\Theta) := \bigg(\mathbb{E}_{\Theta}[\Psi \cdot (M^\Theta_f(t_1 )-M^\Theta_f(t_0 ))]\bigg)^{-},
\end{equation}
where for each $a \in \mathbb{R}$, $a^{-}:=\max\{-a,0\}$. Similarly, for each $B \in (0, \infty)$, define $\Phi_{B}$ by replacing $M^{\Theta}_{f}$ in \eqref{fubu} with $M^{\Theta,B}_{f}$.

Similarly to the proof of Lemma 3.2 in \cite{BC}, we obtain that
(i) $\forall B \in (0,\infty)$, $\Phi_{B}$ is continuous on $\mathcal{P}(\mathcal{Z})$,
(ii) $\sup_{n}{\mathbb{E}}|\Phi(Q^{n})-\Phi_{B}(Q^{n})|\rightarrow 0$,
and (iii) $\tilde{\mathbb{E}}|\Phi(Q)-\Phi_{B}(Q)|\rightarrow 0$ as $B \rightarrow \infty$. For fixed $B>0$, $\Phi_B(Q^n) \rightarrow \Phi_B(Q)$ in distribution. The above properties yield that $\Phi(Q^n) \rightarrow \Phi(Q)$ in distribution as well.
By the definition of $Q^n$, for $\omega \in \Omega$, we obtain
\begin{eqnarray*}
&&\Phi(Q^n_\omega)\\
&=&\Big(\mathbb{E}_{Q^n_\omega}\left[\Psi \cdot \big( M^{Q^n_\omega}_f(t_1)- M^{Q^n_\omega}_f(t_0) \big)\right]\Big)^{-}\\
&=&\Big(\frac{1}{n}\sum_{i=1}^{n}\Psi\big( (\bar X^{i,n}(.,\omega), \rho^{i,n}_\omega, W^i(.,\omega)) \big)\\
   &&\cdot\big(f(t_1,\bar X^{i,n}(t_1 , \omega), W^i(t_1 , \omega) ) -f(t_0,\bar X^{i,n}(t_0 , \omega), W^i(t_0 , \omega))\\
   &&-\int_{t_0}^{t_1}f_s(s,\bar X^{i,n}(s,\omega),W^i(s, \omega))\mathrm{d}s \\
   &&-\int_{t_0}^{t_1}\mathcal {A}^{Q^{n}_{\omega}}_s(f)(s,\bar X^{i,n}(s,\omega), h_i^n(s,\omega), W^i(s, \omega))\mathrm{d}s  \big)\Big)^{-},
\end{eqnarray*}
where $\mathcal {A}^{Q^{n}_{\omega}}$ is defined in \eqref{III-eq3.2} with $\bar{\mu}^n_\omega$ in place of $\nu_\Theta$.

Applying the It\^{o}'s formula, for each $i$, we have
\begin{equation*}\label{V9''}
\begin{aligned}
   &f\big( t_1,\bar X^{i,n}(t_1),W^i(t_1)  \big)
-  f\big(t_0, \bar X^{i,n}(t_0), W^i(t_0)\big)\\
&\ \ \ \  -\int_{t_0}^{t_1}f_s(s,\bar X^{i,n}(s),W^i(s))\mathrm{d}s\\
&\ \ \ \  -\int_{t_0}^{t_1}\mathcal {A}^{Q^{n}}_s(f)\big(s,\bar X^{i,n}(s), h_i^n(s), W^i(s)\big)\mathrm{d}s\\
&= \int_{t_0}^{t_1}
         \langle \nabla_xf(s,\bar X^{i,n}(s), W^i(s)) , \sigma(s,\bar X^{i,n}(s), \bar\mu^n(s))\mathrm{d}W^i(s) \rangle\\
&\ \ \ \  +\int_{t_0}^{t_1}
          \langle \nabla_zf(s,\bar X^{i,n}(s), W^i(s)), \mathrm{d}W^i(s) \rangle \\
&\ \ \ \  -\int_{t_0}^{t_1}
          \langle \nabla_xf(s,\bar X^{i,n}(s), W^i(s)) , \mathrm{d}\bar K^{i,n}_s \rangle.
\end{aligned}
\end{equation*}

Keeping in mind that $\Psi$ is $\mathcal{G}_{t_0}$-measurable and non-negative, we have
\begin{eqnarray*}
&&\mathbb{E}[{\Phi^2(Q^n)}]\\
&=&\mathbb{E}\bigg[ \bigg(\Big (\mathbb{E}_{Q^n_\omega}
        \big[\Psi \cdot\big(M^{Q^n_\omega}_f(t_1 )- M^{Q^n_\omega}_f(t_0) \big) \big]\Big)^{-}\bigg)^{2} \bigg]\\
&=&\mathbb{E}\bigg[ \bigg( \Big( \frac{1}{n}\sum_{i=1}^n \int_{t_0}^{t_1}
        \Psi \cdot \big(\langle \nabla_xf(s,\bar X^{i,n}(s), W^i(s)) , \sigma(s,\bar X^{i,n}(s), \bar\mu^n(s))\mathrm{d}W^i(s) \rangle \\
    &&+\langle \nabla_zf(s,\bar X^{i,n}(s), W^i(s)), \mathrm{d}W^i(s) \rangle
           - \langle \nabla_xf(s,\bar X^{i,n}(s), W^i(s)) , \mathrm{d}\bar K^{i,n}_s \rangle \big) \Big )^{-}  \bigg)^2\bigg]\\
&\leq& \mathbb{E}\bigg[ \bigg( \Big( \frac{1}{n}\sum_{i=1}^n \int_{t_0}^{t_1}
         \Psi\cdot \big(\langle \nabla_xf(s,\bar X^{i,n}(s), W^i(s)) , \sigma(s,\bar X^{i,n}(s), \bar\mu^n(s))\mathrm{d}W^i(s) \rangle \\
    &&+\langle \nabla_zf(s,\bar X^{i,n}(s), W^i(s)), \mathrm{d}W^i(s) \rangle \big) \Big)^{-}\\
    &&+ \Big( -\frac{1}{n}\sum_{i=1}^n \int_{t_0 }^{t_1 } \Psi\cdot \langle \nabla_xf(s,\bar X^{i,n}(s), W^i(s)) , \mathrm{d}\bar K^{i,n}_s \rangle \Big)^{-} \bigg)^2\bigg]\\
&\leq& \frac{1}{n^2}\sum_{i=1}^n\mathbb{E}\bigg[\int_{t_0}^{t_1}
          \Big|\Psi \cdot \Big( \nabla_zf(s,\bar X^{i,n}(s), W^i(s)) \\
    &&+\nabla_xf(s,\bar X^{i,n}(s), W^i(s))\sigma(s,\bar X^{i,n}(s), \bar\mu^n(s))\Big)\Big|^2\mathrm{d}s\bigg] \\
&\stackrel{n\rightarrow\infty}{\longrightarrow}&0,
\end{eqnarray*}
where we have used the fact that
$$-\frac{1}{n}\sum_{i=1}^n \int_{t_0 }^{t_1 } \Psi\cdot \langle \nabla_xf(s,\bar X^{i,n}(s), W^i(s)) , \mathrm{d}\bar K^{i,n}_s \rangle\geq 0.$$
%Thus, for each $(t_0, t_1, \Psi, f) \in \mathcal{I}$, it holds that for $\tilde{\mathbb{P}}$-almost all $\omega \in \tilde{\Omega}$
%\[\mathbb{E}_{Q_\omega}[\Psi \cdot (M^{Q_\omega}_f(t_1)-M^{Q_\omega}_f(t_0)]\geq 0 .\]
Since $\mathcal{I}$ is countable, it follows that for $\tilde{\mathbb{P}}$-almost all $\omega \in \tilde{\Omega}$, all $(t_0, t_1, \Psi, f) \in \mathcal{T}$,
\[\mathbb{E}_{Q_\omega}[\Psi \cdot (M^{Q_\omega}_f(t_1)-M^{Q_\omega}_f(t_0 )]\geq 0 ,\]
which implies that for $\tilde{\mathbb{P}}$-almost all $\omega \in \tilde{\Omega}$, $Q_\omega$ is a weak solution of equation \eqref{II-eq2.7-1} and we complete the proof of the theorem.
\end{proof}

Now we are ready to complete the proof of the upper bound \eqref{lowerbound}. By the definition of $Q^N$, the left-hand side of  \eqref{lowerbound} can be rewritten as
\begin{eqnarray*}\label{K12}
&&\frac{1}{2} \mathbb{E}\bigg[ \frac{1}{N} \sum_{i=1}^N \int_0^T|h_i^N(t)|^2\mathrm{d}t \bigg] + \mathbb{E} [F(\bar\mu^{N})] \\
 &=& \int_\Omega \bigg[\int_{\mathcal{R}_1}\bigg( \frac{1}{2}\int_{\mathcal{R}^{d_1} \times [0,T]}|y|^2 r(\mathrm{d}y \times \mathrm{d}t) \bigg)Q^N_{\omega,\mathcal{R}}(\mathrm{d}r) + F(Q^N_{\omega,\mathcal{X}})  \bigg]\mathbb{P}(\dd \omega).
\end{eqnarray*}
%where $Q^N_{\omega,\mathcal{X}}$, $Q^N_{\omega,\mathcal{R}}$ denote the first and second marginal of $Q^N_\omega \in \mathcal{P}(\mathcal{Z})$, respectively.
%Note that the function $F$ in \eqref{lowerbound} is bounded and continuous.
Noticing that the function $F$ is bounded and continuous, by Lemma \ref{lem1} and Theorem \ref{lem2}, Fatou's Lemma, the Laplace upper bound is immediate.

\subsection{Laplace lower Bound}\label{sec:upper_bound}
In this subsection, we proves the Laplace lower bound \eqref{superbound} in {\bf Step 2}. The proof is very similar to the ones in \cite{BC} and \cite{BudhirajaDF}. For completeness, we give a sketch.

Let $\Theta \in \mathcal{P}_\infty$. We will construct a sequence $\{h^N,N \in \mathbb{N}\}$ with $h^N \in \mathcal {H}_N$ on a common stochastic basis satisfying \eqref{superbound}:
\begin{eqnarray*}
&  &\limsup_{N\to\infty}\Big\{\frac{1}{2}\mathbb{E}\big[\frac{1}{N}\sum^{N}_{i=1}\int^{T}_{0}
    |h^{N}_{i}|^{2}\mathrm{d}t\big]+\mathbb{E}[F(\bar{\mu}^{N})]\Big\}\\
& \leq&
    \frac{1}{2}\int_{\mathcal{R}_1}\int_{\mathbb{R}^{d_{1}}\times[0,T]}|y|^{2}r(\mathrm{d}y\times\mathrm{d}t)    \Theta_{\mathcal{R}}(\mathrm{d}r)+F(
    \Theta_{\mathcal{X}}).
\end{eqnarray*}

Recall $(\bar X, \rho, W)$ from Section \ref{AC}, which is the canonical process on $(\mathcal{Z}, \mathcal{B}(\mathcal{Z}),(\mathcal{G}_t),\Theta)$ and $\bar{v}$ from Section \ref{sec:main_result}, which is the map from $\mathcal{Z}$ to $\mathcal{Z}^{0}$. Then we can disintegrate $\Theta\circ \bar{v}^{-1}$ as
\[  \Theta\circ \bar{v}^{-1}:= \nu_{0}(\mathrm{d}\phi_{0})\Theta_{\mathcal{W}}(\mathrm{d}w)\bar{\lambda}(\mathrm{d}r|\phi_{0},w).  \]

Define a product space $(\Omega_\infty, \mathcal {F}^\infty)$ as the countably infinite product of
$(\mathcal{R}_{1} \times \mathcal{W}, \mathcal{B}(\mathcal{R}_{1} \times \mathcal{W}))$ and write $(r^{\prime}, w^{\prime})$ as the element of $\Omega_\infty$, where $r^{\prime}=(r_1, r_2, \ldots)$ and $w^{\prime}=(w_1, w_2, \ldots)$. For each $i \in \mathbb{N}$, define
\[ W^{i, \infty}(t, (r^{\prime}, w^{\prime})) := w_i(t), \ \ \rho^{i, \infty}(r^{\prime}, w^{\prime}) :=r_i, \]
and
\[ h^\infty_i(t, (r^{\prime}, w^{\prime}))  :=
\int_{\mathbb{R}^{d_1}}y\rho^{i, \infty}_{(r^{\prime}, w^{\prime}), t}(\mathrm{d}y),\ (r^{\prime}, w^{\prime}) \in \Omega_\infty, \ t \in [0,T],\]
where $\rho^{i, \infty}_{(r^{\prime}, w^{\prime}), t}$ is the derivative measure of $\rho^{i, \infty}(r^{\prime}, w^{\prime})$ at time $t$.
Furthermore, for each $N \in \mathbb{N}$, define $\mathbb{P}_{N} \in \mathcal{P}(\Omega_\infty)$ as
\[\mathbb{P}_{N}(\mathrm{d}r,\mathrm{d}w):=\bigotimes_{i=1}^{N}\Theta_{\mathcal{W}}(\mathrm{d}w_{i})\bar{\lambda}(\mathrm{d}r_{i}|x^{i,N},w_{i})
\bigotimes_{i=N+1}^{\infty}\Theta\circ(\rho, W)^{-1}(\mathrm{d}r_{i},\mathrm{d}w_{i}),      \]
and define a $\mathcal{P}(\overline{\mathcal{D}}\times \mathcal{R}_{1} \times \mathcal{W})$-valued random variable by
\[\lambda^{N}(A \times R \times B):=\frac{1}{N}\sum_{i=1}^{N}\delta_{x^{i,N}}(A)\delta_{\rho^{i,\infty}}(R)\delta_{W^{i,\infty}}(B),\]
for any $A\times R \times B \in \mathcal{B}(\overline{\mathcal{D}}\times \mathcal{R}_{1} \times \mathcal{W})$, where $\{x^{i,N}\}$ are the initial values.

By construction, under $\mathbb{P}_{N}$, $\{W^{i, \infty} , i \in \{1,\ldots, N\} \}$ are independent Wiener processes. By Assumption \ref{Ass2.1}(A2), we obtain that
\begin{equation}\label{Weakuniqueness5}
\mathbb{P}_{N}\circ (\lambda^{N})^{-1}\rightarrow \delta_{\Theta \circ \bar{v}^{-1}}
\end{equation}
and
\begin{eqnarray}\label{Weakuniqueness6}
&&  \limsup_{N \rightarrow \infty} \mathbb{E}_{\mathbb{P}_{N}}\bigg[\frac{1}{N}\sum_{i=1}^{N}\int_0^T|h^{\infty}_{i}(t)|^2\mathrm{d}t\bigg]\\
&=&  \limsup_{N \rightarrow \infty} \frac{1}{N}\sum_{i=1}^{N}\int_{\mathcal{R}_{1}\times\mathcal{W}}
\int_0^T|\int_{\mathbb{R}^{d_1}}yr_t(\mathrm{d}y)|^2\mathrm{d}t\bar{\lambda}(\mathrm{d}r|x^{i,N},w)\Theta_{\mathcal{W}}(\mathrm{d}w)\nonumber\\
&=&  \mathbb{E}_{\Theta}\bigg[ \int_{0}^{T}|\int_{\mathbb{R}^{d_1}}y\rho_{t}( \mathrm{d}y)|^2 \mathrm{d}t \bigg]
      \leq \mathbb{E}_{\Theta}\bigg[ \int_{\mathbb{R}^{d_1} \times [0,T] }|y|^2 \rho(\mathrm{d}y \times \mathrm{d}t) \bigg]<\infty.\nonumber
\end{eqnarray}

In analogy with \eqref{V2}, for each $N \in \mathbb{N}$, define a $\mathcal{P}(\mathcal{Z})$-valued random variable by
\[\tilde {Q}^N (A \times R \times B):= \frac{1}{N} \sum_{i=1}^N \delta_{\tilde X^{i,N}}(A) \cdot \delta_{\rho^{i,\infty}}(R)
\cdot \delta_{ W^{i, \infty}}(B),\]
for any $A\times R \times B \in \mathcal{B}(\mathcal{Z})$, where $(\tilde X^{1,N},\ldots, \tilde X^{N,N})$ is the solution of the system \eqref{II-eq2.3} with $h^{N}=(h^{\infty}_1,\ldots,h^{\infty}_N)$ and $\tilde{\mu}^N(t)=\frac{1}{N}\sum_{i=1}^{N}\delta_{\tilde{X}^{i,N}(t)}$ for each $t \in [0,T]$.

By \eqref{Weakuniqueness6} and Lemma \ref{tightness}, we know that $\{\tilde{Q}^N,N \in \mathbb{N}\}$ is tight. Let $\tilde{Q}$ be a limit point of $\{\tilde{Q}^N,N \in \mathbb{N}\}$ defined on some probability space $(\tilde{\Omega},  \tilde{\mathcal {F}}, \tilde{\mathbb{P}})$. According to Theorem \ref{lem2} and its proof, $\tilde{Q}_\omega \in \mathcal{P}_{\infty}$ for $\tilde{\mathbb{P}}$-almost all $\omega \in \tilde{\Omega}$. Moreover, by \eqref{Weakuniqueness5}, we derive that for  $\tilde{\mathbb{P}}$-almost all $\omega \in \tilde{\Omega}$, $\tilde{Q}_\omega \circ \bar{v}^{-1} = \Theta \circ \bar{v}^{-1}$. Therefore, by Assumption \ref{Ass2.1}(A4), for  $\tilde{\mathbb{P}}$-almost all $\omega \in \tilde{\Omega}$, $\tilde{Q}_\omega =  \Theta$. In combination with \eqref{Weakuniqueness6} and $F$ in \eqref{superbound} is bounded and continuous, the Laplace lower bound is immediate.
\vskip 0.4cm
\noindent{\bf Acknowledgement}. This work is partially supported by National Key R\&D
Program of China (No.2022YFA1006001), National Natural Science Foundation of China (Nos. 12131019, 11971456, 11721101)

\bibliographystyle{amsplain}

%\bibliography{RILDP}

\end{document}